\documentclass[sn-mathphys,Numbered]{sn-jnl}% Math and Physical Sciences Reference Style
\usepackage{graphicx}%
\usepackage{multirow}%
\usepackage{amsmath,amssymb,amsfonts}%
\usepackage{amsthm}%
\usepackage{mathrsfs}%
\usepackage[title]{appendix}%
\usepackage{xcolor}%
\usepackage{textcomp}%
\usepackage{manyfoot}%
\usepackage{booktabs}%
\usepackage{algorithm}%
\usepackage{algorithmicx}%
\usepackage{algpseudocode}%
\usepackage{listings}%

\theoremstyle{thmstyleone}%
\newtheorem{theorem}{Theorem}%  meant for continuous numbers
\newtheorem{proposition}[theorem]{Proposition}% 
\newtheorem{corollary}[theorem]{Corollary}% 
\newtheorem{lemma}[theorem]{Lemma}% 

\theoremstyle{thmstyletwo}%
\newtheorem{example}{Example}%
\newtheorem{remark}{Remark}%

\theoremstyle{thmstylethree}%
\newtheorem{definition}{Definition}%

\begin{document}
	
	\title[Stochastic two-scale convergence in the mean in Orlicz-Sobolev's spaces and applications to the homogenization of an integral functional]{Stochastic two-scale convergence in the mean in Orlicz-Sobolev's spaces and applications to the homogenization of an integral functional}
	
	
	\author[1]{\fnm{\textsc{Dongho}} \sur{\textsc{Joseph}}}\email{joseph.dongho@fs.univ-maroua.cm}
	\equalcont{These authors contributed equally to this work.}
	
	\author*[2]{\fnm{\textsc{Fotso Tachago}} \sur{\textsc{Joel}}}\email{fotsotachago@yahoo.fr}
	\equalcont{These authors contributed equally to this work.}
	
	\author[3]{\fnm{\textsc{Tchinda Takougoum}} \sur{\textsc{Franck}}}\email{takougoumfranckarnold@gmail.com}  
	\equalcont{These authors contributed equally to this work.}

	\affil[1]{\orgdiv{Department of Mathematics and Computer Science}, \orgname{University of Maroua}, \orgaddress{\street{Domayo}, \city{Maroua}, \postcode{P.O. Box 814}, \country{Cameroon}}}
	
	\affil*[2]{\orgdiv{Department of Mathematics}, \orgname{Higher Teachers Trainning College,University of Bamenda}, \orgaddress{\street{Bambili}, \city{Bamenda}, \postcode{P.O. Box 39}, \country{Cameroon}}}
	
	\affil[3]{\orgdiv{Department of Mathematics and Computer Science}, \orgname{University of Maroua}, \orgaddress{\street{Domayo}, \city{Maroua}, \postcode{P.O. Box 814}, \country{Cameroon}}}

	
	\abstract{In this paper, we study the stochastic homogenization for a family of integral functionals with convex and nonstandard growth integrands defined on Orlicz-Sobolev's spaces. One  fundamental in this topic is to extend the classical compactness results of the two-scale convergence in the mean method to this type of spaces.  Moreover, it is shown by the two-scale convergence in the mean method that the sequence of minimizers of a class of highly oscillatory minimizations problems involving convex functionals converges to the minimizers of a homogenized problem with a suitable convex function.\footnotetext{This work is supported by MINESUP/University of Maroua}}

	
	
	%35B27 EDP? propriete qualitative des solutions, homogeneisation des edp
	%37A05 Dynamical aspects of measure -preserving transformations
	%46E30 / linear functions spaces : spaces of measurable function (Orlicz) 
	%35J20 Variational method  for second order elliptic equations
	\maketitle
	
	%%%%%%%%%%%%%%%%%%%%%%%%%%%%%%%%%%%%%%%%%%%%%%%%%%%%%%%%%%%%%%%%%%%%%%%%%%%%%%%%%%%%%
	%%%%%%%%%%%%%%%%%%%%    1ere PARTIE   %%%%%%%%%%%%%%%%%%%%%%%%%%%%%%%%%%%%%%%%%%%%%%%
	
	\section{Introduction} \label{sect1} 
	\par Let $\Phi : [0,+\infty[ \rightarrow [0,+\infty[ $ be a Young function of class   $\Delta_{2}\cap \Delta'$ (see Definition \ref{c1deltaC}), such that its Fenchel's conjugate $\widetilde{\Phi}$ is also a Young function of class $\Delta_{2}\cap \Delta'$, and $(\Omega, \mathscr{M}, \mu)$ be a measure space with probability measure $\mu$. 
	Let $(\omega,\lambda) \rightarrow f(\omega,\lambda)$ be a function from $\Omega\times \mathbb{R}^{N}$ into $\mathbb{R}$ satisfying the following properties:
	\begin{itemize}
		\item[(H$_{1}$)] $f(\cdot, \lambda)$ is measurable for all $\lambda \in \mathbb{R}^{N}$ ;
		\item[(H$_{2}$)] $f(\omega,\cdot)$ is strictly convex for almost all $\omega \in \Omega$ ;
		\item[(H$_{3}$)] There are two constants $c_{1}, \, c_{2} > 0$ such that 
		\begin{equation}
			c_{1} \Phi(|\lambda|) \leq f(\omega,\lambda) \leq c_{2}(1+ \Phi(|\lambda|))
		\end{equation}
		for all $\lambda \in \mathbb{R}^{N}$ and for almost all $\omega \in \Omega$.      
	\end{itemize}
	Let $Q$ be a bounded open set in $\mathbb{R}^{N}_{x}$ (the space $\mathbb{R}^{N}$ of variables $x = (x_{1}, \cdots , x_{N})$) and $\{ T(x) : \Omega \rightarrow \Omega, \, x \in \mathbb{R}^{N} \}$ be a fixed $N$-dimensional dynamical system on $\Omega$, with respect to a fixed probability measure $\mu$ on $\Omega$ invariant for $T$.  For each given $\epsilon > 0$, let $F_{\epsilon}$ be the integral functional defined by 
	\begin{equation}\label{ras1}
		F_{\epsilon}(v) = \iint_{Q\times \Omega} f\left( T(\epsilon^{-1}x)\omega, \, Dv(x,\omega) \right)dxd\mu, \quad v \in W^{1}_{0}L^{\Phi}_{D_{x}}(Q; L^{\Phi}(\Omega)),
	\end{equation} 
	where $D$ denotes the usual gradient operator in $Q$, i.e. $D = D_{x}= \left(\frac{\partial}{\partial x_{i}} \right)_{1\leq i\leq N}$, $f$ is a random integrand satisfying above conditions $(H_{1})-(H_{3})$ and  $W^{1}_{0}L^{\Phi}_{D_{x}}(Q; L^{\Phi}(\Omega))$ is an Orlicz-Sobolev space which will be specified later  (see (\ref{c1orli1S}) and comment below).  We are interested in the homogenization of the sequence of solutions of the problem
	\begin{equation}\label{ras2}
		\min \left\{ F_{\epsilon}(v) : v \in W^{1}_{0}L^{\Phi}_{D_{x}}(Q; L^{\Phi}(\Omega)) \right\}.
	\end{equation} 
	i.e. the analysis of the asymptotic behaviour of minimizers of functionals $F_{\epsilon}$
	when $\epsilon \rightarrow 0$. With considerations to principles of physics (see \cite{abda,gambi,jikov}), e.g. in elasticity theory, the term $F_{\epsilon}(v)$ can be viewed as the energy under a deformation $v$ of an elastic body whose microstructures behave randomly in some scales. Let us point out that, the functionals $F_{\epsilon}$ have been studied by many authors. In \cite{tacha1}, Fotso  Tachago and Nnang study the periodic homogenization in the Orlicz-Sobolev's  spaces when the functionals $F_{\epsilon}$ are of type $\int_{Q} f(\frac{x}{\epsilon}, Dv(x))\,dx$, where $v \in W^{1}_{0}L^{\Phi}(Q)$  (see Section \ref{sect2} for the notations).  %\cite[Page 179]{tacha1}
	For more details on homogenization via the periodic two-scale convergence in the Orlicz-Sobolev's spaces, we refer to \cite{tacha2,tacha3,tacha4,tacha5}.  In \cite{sango}, the general case, when the functional $F_{\epsilon}$ are of type   $\iint_{Q\times\Omega} f(x,T(\frac{x}{\epsilon}\omega), \frac{x}{\epsilon^{2}}, Dv(x,\omega))\,dx\,d\mu$ has been studied for the test functions taking values in classical Sobolev's spaces, i.e. $v \in L^{p}(\Omega, W_{0}^{1,p}(Q))$  (the space of functions $v \in L^{p}(\Omega\times Q)$ such that for a.e. $\omega\in \Omega$, $v(\omega, \cdot) \in  W_{0}^{1,p}(Q)$ and $\int_{\Omega} \|v(\omega,\cdot)\|^{p}_{W_{0}^{1,p}(Q)} < +\infty$, with $p\geq 1$). %\cite[Page 111]{sango}
	For more details we refer, e.g. to \cite{abda,bia,jikov,wou}.
	\par Our contribution in this paper is to study the limiting behaviour, via the stochastic two-scale converge in the mean technique of $F_{\epsilon}$ in their form (\ref{ras1}), defined on the Orlicz-Sobolev's spaces $W^{1}_{0}L^{\Phi}_{D_{x}}(Q; L^{\Phi}(\Omega))$ which is more general than the classical Sobolev space $ L^{p}(\Omega, W_{0}^{1,p}(Q))$. 
	The extension to the Orlicz-Sobolev's spaces is also motivated by the fact that the two-scale convergence method introduced by Nguetseng \cite{nguet1} and later developed by Allaire \cite{allair1} have been widely adopted in homogenization of PDEs in classical Sobolev spaces neglecting materials where microstructure cannot be conveniently captured by modeling exclusively by means of those spaces. Moreover it is well know (see, e.g. \cite{mignon2}) that there exist problems whose solution must naturally belong not to the classical Sobolev spaces but rather to the Orlicz-Sobolev spaces.  \\
	In the framework of stochastic homogenization problems, Bourgeat, Mikeli\`{c} and Wright (see \cite{bourgeat}) prove the following two-scale in the mean convergence theorem for Sobolev's spaces: \textit{ for $p=2$, any bounded sequence $(u_{\epsilon})_{\epsilon>0}$ in $L^{p}(\Omega, W^{1,p}(Q))$ admits a subsequence still denoted by $(u_{\epsilon})$  such that, as $\epsilon\to 0$, one has 
		\begin{equation*}
			\left\{\begin{array}{l}
				u_{\epsilon}  \rightarrow  u_{0} \; stoch.\;mean \; in \; L^{p}(Q\times\Omega)-weak \, 2s,  \\
				\\
				\displaystyle 	\iint_{Q\times\Omega} Du_{\epsilon}(x,\omega) \phi(x, T(\epsilon^{-1}x)\omega)dxd\mu \\
				\longrightarrow  \displaystyle  \iint_{Q\times\Omega} (Du_{0}(x,\omega) + \overline{D}_{\omega}u_{1}(x,\omega))\phi(x, \omega)dxd\mu, 
			\end{array} 
			\right.
		\end{equation*}
		for all $\phi \in [\mathcal{C}^{\infty}_{0}(Q)\otimes\mathcal{C}^{\infty}(\Omega)]^{N}$, where $u_{0} \in W^{1,p}(Q, I_{nv}^{p}(\Omega))$, $u_{1} \in L^{p}(Q, W^{1,p}_{\#}(\Omega))$ with $I_{nv}^{p}(\Omega)$ denote the subspace of $L^{p}(\Omega)$ consisting the functions that are invariant for the dynamical system $T$ and $W^{1,p}_{\#}(\Omega)$ denote the completion of $\mathcal{C}^{\infty}(\Omega)$ with respect to the norm $\|f\|_{W^{1,p}_{\#}(\Omega)} = \left(  \sum_{i=1}^{N}\|\overline{D}_{i,p}f\|^{p}_{L^{p}(\Omega)} \right)^{1/p}$ for $f \in W^{1,p}_{\#}(\Omega)$, $\overline{D}_{i,p}$ and $\overline{D}_{\omega}$ being the stochastic derivative and stochastic gradient respectively in $L^{p}$-spaces (see Subsection \ref{c1sub3S}). }\\
	The generalization of this result to Orlicz-Sobolev's spaces (see Theorem \ref{lem29}) is fundamental to the proof of the main result of this work. Precisely, we show (see Theorem \ref{lem37}) that: \textit{ the family of minimizers of (\ref{ras2}), $(u_{\epsilon})_{\epsilon>0}$, satisfies, as $\epsilon\to 0$
		\begin{equation*}
			\left\{\begin{array}{l}
				u_{\epsilon} \rightarrow  u_{0} \; stoch.\;mean \; in \; L^{\Phi}(Q\times\Omega)-weak \, 2s,   \\
				\\
				\displaystyle \iint_{Q\times\Omega} Du_{\epsilon}(x,\omega) \phi(x, T(\epsilon^{-1}x)\omega)dxd\mu \\
				\longrightarrow  \displaystyle  \iint_{Q\times\Omega}(Du_{0}(x,\omega) + \overline{D}_{\omega}u_{1}(x,\omega))\phi(x, \omega)dxd\mu, 
			\end{array} 
			\right.
		\end{equation*}
		for all $\phi \in [\mathcal{C}^{\infty}_{0}(Q)\otimes\mathcal{C}^{\infty}(\Omega)]^{N}$, where  the couple $(u_{0}, u_{1})$ is the minimizer in the space \\ $W^{1}_{0}L^{\Phi}_{D_{x}}(Q; I^{\Phi}_{nv}(\Omega)) \times L^{\Phi}(Q, W^{1}_{\#}L^{\Phi}(\Omega))$ of the global homogenized functional 
		\begin{equation*}
			(v_{0}, v_{1}) \rightarrow \iint_{Q\times\Omega} f(\omega, Dv_{0} + \overline{D}_{\omega}v_{1}) dxd\mu,
		\end{equation*}
		where the spaces $I_{nv}^{\Phi}(\Omega)$ and $W^{1}_{\#}L^{\Phi}(\Omega)$ will be specified later  (see Subsection \ref{c1sub3S}).  }  
	\par The paper is organized as follows. Section \ref{sect2} deals with some preliminary results on Orlicz-Sobolev's spaces and dynamical systems. In section \ref{sect3}, we extend the proofs of the main stochastic two-scale convergence in the mean results from $L^{p}$-spaces to Orlicz's spaces. Finally, in section \ref{sect4}, we sketch out the existence and uniqueness of minimizers for (\ref{ras2}) and we study the stochastic homogenization of problem for (\ref{ras2}).

	%%%%%%%%%%%%%%%%%%%%%%%%%%%%%%%%%%%%%%%%%%%%%%%%%%%%%%%%%%%%%%%%%%%%%%%%%%%%%%%%%%%%%
	%%%%%%%%%%%%%%%%%%%%    2ere PARTIE   %%%%%%%%%%%%%%%%%%%%%%%%%%%%%%%%%%%%%%%%%%%%%%%
	
	\section{Preliminary on Orlicz-Sobolev's spaces and dynamical systems}\label{sect2}
	
We recall some notations that we will use throughout this work.
\begin{itemize}
	\item 	$Q$ is a bounded open set of $\mathbb{R}^{N}$, integer $N>1$.
	\item $\mathcal{D}(Q)=\mathcal{C}^{\infty}_{0}(Q)$ is the vector space of smooth functions with compact support in $Q$. 
	\item	$\mathcal{C}^{\infty}(Q)$ is the vector space of smooth functions on $Q$. 
	\item	$\mathcal{K}(Q) $ is the vector space of continuous functions with compact support in $Q$. 
	\item	$(\Omega, \mathscr{M}, \mu)$ is a measure space with probability measure $\mu$. 
	\item	$\{T(x): \Omega\rightarrow \Omega \, , \, x \in \mathbb{R}^{N} \}$ is a $N$-dimensional dynamical system on $\Omega$.  
	\item	$\Phi$ is a Young function and $\widetilde{\Phi}$ its conjugate. 
	\item	$L^{\Phi}(\Omega)$ and $L^{\Phi}(Q)$ are Orlicz spaces of functions on $\Omega$ and $Q$, respectively.  
	\item	$W^{1}L^{\Phi}(Q) = \{ v \in L^{\Phi}(Q) \, : \, \frac{\partial v}{\partial x_{i}} \in L^{\Phi}(Q), \; 1 \leq i \leq N \}$, where derivatives are taken in the distributional sense, is an Orlicz-Sobolev's space on $Q$. 
	\item	$W^{1}_{0}L^{\Phi}(Q)$ the set of functions in $W^{1}L^{\Phi}(Q)$ with zero boundary condition on $Q$. 
	\item	$W^{1}L^{\Phi}(Q\times\Omega) = \{ v \in L^{\Phi}(Q\times\Omega) \, : \, \frac{\partial v}{\partial x_{i}} \in L^{\Phi}(Q\times\Omega) \; \textup{and} \; D_{i,\Phi}v \in L^{\Phi}(Q\times\Omega), \; 1 \leq i \leq N \}$, where derivatives are taken in the distributional sense, is an Orlicz-Sobolev's space on $Q\times\Omega$.
	\item	$W^{1}L^{\Phi}_{D_{x}}(Q; L^{\Phi}(\Omega)) = \{ v \in L^{\Phi}(Q\times\Omega) \, : \, \frac{\partial v}{\partial x_{i}} \in L^{\Phi}(Q\times\Omega), \; 1 \leq i \leq N \}$, where derivatives are taken in the distributional sense, is an Orlicz-Sobolev's space on $Q\times\Omega$.
	\item	$W^{1}_{0}L^{\Phi}_{D_{x}}(Q; L^{\Phi}(\Omega))$ the set of functions in $W^{1}L^{\Phi}_{D_{x}}(Q; L^{\Phi}(\Omega))$ with zero boundary condition on $Q$.
\end{itemize}
	
	\subsection{Young's function}
	
		All definitions and results recalled here are classical and can be found in \cite{adam,tacha1,tacha3}. 
		Let $\Phi : [0,+\infty) \rightarrow  [0,+\infty)$ be a Young function, that is, 
		$\Phi$ is continuous, convex ,
		$\Phi(t) > 0$ for $t > 0$,
		$\frac{\Phi(t)}{t} \to 0$ as $t\to 0$ and $\frac{\Phi(t)}{t} \to \infty$ as $t\to \infty$. Then
		$\Phi$ has an integral representation under form
		$
		\Phi(t) = \int_{0}^{t} \phi(\tau)\, d\tau,
		$
		where $\phi : [0, \infty) \rightarrow  [0, \infty)$ is nondecreasing, right continuous, with $\phi(0) = 0$, $\phi(t) > 0$ if $t>0$ and $\phi(t)\rightarrow \infty$ if $t\rightarrow \infty$. We denote $\widetilde{\Phi}$
		the Fenchel's conjugate (or complementary function) of the Young function $\Phi$, that is, $\widetilde{\Phi}(t) = \sup_{s\geq 0} \left( st - \Phi(s) \right) \; (t\geq 0)$, then $\widetilde{\Phi}$ is also a Young function. Given two Young functions $\Phi$ and $\Psi$, we say that $\Phi$ dominates $\Psi$ (denoted by $\Phi \succ \Psi$ or $\Psi \prec \Phi$) near infinity if there are $k>1$ and $t_{0}>0$ such that 
		\begin{equation*}
			\Psi(t) \leq \Phi(kt), \quad \forall t> t_{0}.
		\end{equation*}
		With this in mind, it is well known that if $\displaystyle \lim_{t\to +\infty} \frac{\Phi(t)}{\Psi(t)} = +\infty$ then $\Phi$ dominates $\Psi$ near infinity. Let us recall some important properties about Young's functions.  
	\begin{definition}\label{c1deltaC}\cite{adam} 
		\begin{enumerate}
			\item A Young function $\Phi$ is said to satisfy the $\Delta_{2}$-condition or $\Phi$ belongs to the class $\Delta_{2}$ at $\infty$, which is written as  $\Phi \in \Delta_{2}$, if there exist constants  $t_{0} > 0$, $k > 2$ such that  
			\begin{equation*}\label{c1eq1r}
				\Phi(2t) \leq k \Phi(t),
			\end{equation*}
			for all $t \geq t_{0}$.
			\item A Young function $\Phi$ is said to satisfy the $\nabla_{2}$-condition or $\Phi$ belongs to the class $\nabla_{2}$ at $\infty$, which is written as  $\Phi \in \nabla_{2}$, if there exist constants $t_{0} > 0$, $c > 2$ such that  
			\begin{equation*}
				\Phi(t) \leq \dfrac{1}{2c} \Phi(ct),
			\end{equation*}
			for all $t \geq t_{0}$.
			\item A Young function $\Phi$ is said to satisfy the $\Delta'$-condition (or $\Phi$ belongs to the class $\Delta'$) denoted by $\Phi \in \Delta'$ if there exists $\beta > 0$ such that  
			\begin{equation*}
				\Phi(ts) \leq \beta\, \Phi(t)\Phi(s), \quad \forall t,s \geq 0.
			\end{equation*}
		\end{enumerate}
	\end{definition}
	\begin{lemma}\cite{tacha3} \\
		Let  $\Phi$ be a Young function of class $\Delta_{2}$. Then there are $k> 0$ and $t_{0} \geq 0$ such that 
		\begin{equation}\label{lem11}
			\widetilde{\Phi}(\phi(t)) \leq k \, \Phi(t) \quad \textup{for} \, \textup{all} \, t \geq t_{0}.
		\end{equation}	
	\end{lemma}
	\begin{lemma}\cite{tacha3} \\
		Let $t\rightarrow \Phi(t) = \int_{0}^{t} \phi(\tau) d\tau$ be a Young function, and let $\widetilde{\Phi}$ be the Fenchel's conjugate of $\Phi$. Then one has
		\begin{equation}\label{lem10} 
			\left\{ \begin{array}{l}
				\dfrac{t\,\phi(t)}{\Phi(t)} \geq 1 \quad (\textup{resp.} \, > 1 \; \textup{if} \, \phi \, \textup{is} \, \textup{strictly} \; \textup{increasing})  \\
				
				\\
				\widetilde{\Phi}(\phi(t)) \leq t\,\phi(t) \leq \Phi(2t)
			\end{array}\right.
		\end{equation}
		for all $t >0$.
	\end{lemma}
	We give now some examples of Young functions.
	\begin{example}
		The function $t \rightarrow \frac{t^{p}}{p}$, ($p > 1$) is a Young function which satisfy $\Delta_{2}$-condition and $\nabla_{2}$-condition. Its Fenchel's conjugate is a Young function $t \rightarrow \frac{t^{q}}{q}$, where $\frac{1}{p} + \frac{1}{q} = 1$. The function $t \rightarrow t^{p}\ln(1+t)$, ($p \geq 1$) is a Young function that satisfies $\Delta_{2}$-condition, while the Young function $t \rightarrow t^{\ln t}$ and $t \rightarrow e^{t^{r}} - 1$, ($r >0$) are not of class $\Delta_{2}$. However, for $r=1$, the Young function $t \rightarrow e^{t} - 1$ satisfy $\Delta'$-condition.   
	\end{example}	
	
	\subsection{Orlicz spaces}
	
	Let $Q$ be a bounded open set in $\mathbb{R}^{N}$ (integer $N\geq 1$), and let $\Phi$ be a Young function. The Orlicz-class $\widehat{L}^{\Phi}(Q)$ is defined to be the space of all measurable functions $u: Q\rightarrow \mathbb{R}$ such that 
	\begin{equation*}
		\int_{Q}^{} \Phi\left(\dfrac{|u(x)|}{\delta} \right)dx < +\infty,
	\end{equation*}
	for some $\delta=\delta(u) >0$. \\
	The Orlicz-space $L^{\Phi}(Q)$ is the smallest vector space containing the Orlicz-class.\\
	On the space $L^{\Phi}(Q)$ we define two equivalent norms:
	\begin{itemize}
		\item[i)] the Luxemburg norm, 
		\begin{equation*}
			\lVert u \rVert_{L^{\Phi}(Q)} = \inf \left\{ \delta>0 \; : \; \int_{Q}^{} \Phi\left(\dfrac{|u(x)|}{\delta} \right)dx  \leq 1 \right\}, \quad \forall u \in L^{\Phi}(Q),
		\end{equation*}
		\item[ii)] the Orlicz norm, 
		\begin{equation*}
			\begin{array}{rcc}
				\lVert u \rVert_{L^{(\Phi)}(Q)}& =& \sup \left\{ \left| \displaystyle  \int_{Q} u(x)v(x) dx \right| \; : \; v \in L^{\widetilde{\Phi}}(Q) \; \textup{and} \; \|v\|_{L^{\widetilde{\Phi}}(Q)} \leq 1 \right\}, \\
				& &  \\
				& &  \forall u \in L^{\Phi}(Q),
			\end{array}
		\end{equation*}
	\end{itemize}
	which makes it a Banach space.  \\
	Now we will give some properties of Orlicz spaces and we will refer to \cite{adam} for more details.
	Assume that $\Phi \in \Delta_{2}$. Then:
	\begin{itemize}
		\item[i)] $\mathcal{D}(Q)$ is dense in $L^{\Phi}(Q)$
		\item[ii)] $L^{\Phi}(Q)$ is separable and reflexive whenever $\widetilde{\Phi}\in \Delta_{2}$,
		\item[iii)] the dual of $L^{\Phi}(Q)$ is identified with $L^{\widetilde{\Phi}}(Q)$, and the dual norm on $L^{\widetilde{\Phi}}(Q)$ is equivalent to $\lVert \cdot \rVert_{L^{\widetilde{\Phi}}(Q)}$,
		\item[iv)] given $u \in L^{\Phi}(Q)$ and $v \in L^{\widetilde{\Phi}}(Q)$ the product $uv$ belongs to $L^{1}(Q)$ with the generalized H\"{o}lder's inequality 
		\begin{equation*}
			\left| \int_{Q} u(x)v(x)dx \right| \leq 2\, \lVert u \rVert_{L^{\Phi}(Q)} \, \lVert v \rVert_{L^{\widetilde{\Phi}}(Q)},
		\end{equation*}
		\item[v)] given $v \in L^{\Phi}(Q)$ the linear functional $L_{v}$ on $L^{\widetilde{\Phi}}(Q)$ defined by 
		\begin{equation*}
			L_{v}(u) = \int_{Q} u(x)v(x) dx, \quad u \in L^{\widetilde{\Phi}}(Q),
		\end{equation*} 
		belongs to the dual $[L^{\widetilde{\Phi}}(Q)]'$ with $\lVert v \rVert_{L^{(\Phi)}(Q)} \leq \lVert L_{v} \rVert_{[L^{\widetilde{\Phi}}(Q)]'} \leq 2 \lVert v \rVert_{L^{(\Phi)}(Q)}$,
		\item[vi)] $L^{\Phi}(Q) \hookrightarrow L^{1}(Q) \hookrightarrow	L^{1}_{\textup{loc}}(Q) \hookrightarrow \mathcal{D'}(Q)$,
		\item[vii)] given two Young's functions $\Phi$ and $\Psi$, we have the continuous embbeding \\ $L^{\Phi}(Q) \hookrightarrow L^{\Psi}(Q)$ if and only if $\Phi \succ \Psi$ near infinity,
		\item[viii)] the product space $L^{\Phi}(Q)^{N}= L^{\Phi}(Q)\times L^{\Phi}(Q) \times \cdots \times L^{\Phi}(Q)$, ($N$-times), is endowed with the norm 
		\begin{equation*}
			\lVert \textbf{v} \rVert_{L^{(\Phi)}(Q)^{N}} = \sum_{i=1}^{N} \lVert v_{i} \rVert_{L^{(\Phi)}(Q)}, \quad \textbf{v}=(v_{i}) \in L^{\Phi}(Q)^{N}.
		\end{equation*}
		\item[ix)] \label{property9}  If $Q_{1} \subset \mathbb{R}^{N_{1}}$ and $Q_{2} \subset \mathbb{R}^{N_{2}}$ are two bounded open sets with $N_{1}+N_{2}=N$, and if $u \in L^{\Phi}(Q_{1}\times Q_{2})$,
		then for almost all $x_{1}\in Q_{1}$, $u(x_{1}, \cdot) \in L^{\Phi}(Q_{2})$. If in addition $\widetilde{\Phi} \in \Delta'$ associate with a constant $\beta$, then the function $u$ belongs to $L^{\Phi}(Q_{1}, L^{\Phi}(Q_{2}))$, with	
		\begin{equation}\label{be1}
			\|u\|_{L^{\Phi}(Q_{1}, L^{\Phi}(Q_{2}))} \leq \iint_{Q_{1}\times Q_{2}} \Phi(|u(x_{1},x_{2})|) dx_{1}dx_{2} + \beta.
		\end{equation}
	\end{itemize}
	
	\subsection{Measurable dynamics and Ergodic theory in Orlicz's spaces}\label{c1sub3S}
	
	Let $(\Omega, \mathscr{M}, \mu)$ be a measure space with probability measure $\mu$. We define an $N$-dimensional dynamical system on $\Omega$ as a family $\{ T(x) : x \in \mathbb{R}^{N} \}$	of invertible maps,
	\begin{equation*}
		T(x) : \Omega \longrightarrow \Omega
	\end{equation*}
	such that for each $x \in \mathbb{R}^{N}$ both $T(x)$ and $T(x)^{-1}$ are measurable, and such that the following properties hold: 
	\begin{itemize}
		\item[(i)]\textit{(group property)} $T(0)=$ identity map on $\Omega$ and for all $x_{1},x_{2} \in \mathbb{R}^{N}$,
		\begin{equation*}
			T(x_{1}+x_{2})=T(x_{1})\circ T(x_{2}),
		\end{equation*}
		\item[(ii)]\textit{(invariance)} for each $x \in \mathbb{R}^{N}$, the map $T(x) : \Omega \rightarrow \Omega$ is measurable on $\Omega$ and $\mu$-measure preserving, i.e. $\mu(T(x)F)=\mu(F)$, for all $F \in \mathscr{M}$,
		\item[(iii)]\textit{(measurability)} for each $F \in \mathscr{M}$, the set $\{ (x,\omega) \in \mathbb{R}^{N}\times \Omega : T(x)\omega \in F \}$ is measurable with respect to the product $\sigma$-algebra $\mathscr{L}\otimes \mathscr{M}$, where $\mathscr{L}$ is the $\sigma$-algebra of Lebesgue measurable sets. 
	\end{itemize}
	If $\Omega$ is a compact topological space, by a continuous $N$-dynamical system on $\Omega$, we mean any family of mappings $\{ T(x) : \Omega \rightarrow \Omega,\; x \in \mathbb{R}^{N} \}$  satisfying the above group property $(i)$ and the following condition :
	\begin{itemize}
		\item[(iv)] \textit{(continuity)} the mapping $(x, \omega) \mapsto T(x)\omega$ is continuous from $\mathbb{R}^{N}\times \Omega$ to $\Omega$.
	\end{itemize} 
		We have the following proposition which is inspired by Pankov's book \cite[Proposition 3.11]{pankov}.  
	\begin{proposition}(group of isometries)\\
		Let $\Phi \in \Delta_{2}$ be a Young function. An $N$-dimensional dynamical system $\{ T(x) : \Omega\rightarrow \Omega\; ; \; x \in \mathbb{R}^{N} \}$ induces an $N$-parameter group of isometries $\{ U(x) : L^{\Phi}(\Omega)\rightarrow L^{\Phi}(\Omega)\; ; \; x \in \mathbb{R}^{N} \}$ defined by
		\begin{equation}\label{group1S}
			\left(U(x)f  \right)(\omega) = f(T(x)\omega), \quad \forall\, f \in L^{\Phi}(\Omega)
		\end{equation}
		which is strongly continuous, i.e.
		\begin{equation}\label{lem38}
			\lim_{x \to 0} \|U(x)f - f\|_{L^{\Phi}(\Omega)}=0, \quad \forall\, f \in L^{\Phi}(\Omega).
		\end{equation}
	\end{proposition}
	\begin{proof}
		Firstly we will show that if $f \in L^{\Phi}(\Omega)$ then $U(x)f \in L^{\Phi}(\Omega)$, for all $x \in \mathbb{R}^{N}$. Since $f$ and $T$ are measurable on $\Omega$ then $U(x)f= f\circ T(x)$ is also measurable on $\Omega$. Let now $\mu_{T}$ the range measure of $\mu$ by $T$. Since the measure $\mu$ is invariant under the dynamical system $T$, we have 
		\begin{equation*}
			\begin{array}{rcl}
				\displaystyle 	\int_{\Omega} \Phi(|(U(x)f)(\omega)|) d\mu & = & \displaystyle\int_{\Omega} \Phi(|f(T(x)\omega)|) d\mu \\
				& = & \displaystyle \int_{\Omega} \Phi(|f(\omega)|) d\mu_{T} \\
				& = & \displaystyle \int_{\Omega} \Phi(|f(\omega)|) d\mu \\
				& < & \infty,
			\end{array}
		\end{equation*}
		thus $U(x)f \in L^{\Phi}(\Omega)$ and it results that $\|U(x)f\|_{\Phi,\Omega} = \|f\|_{\Phi,\Omega}$. \\
		Secondly, we will prove (\ref{lem38}). For that, it suffices to verify (\ref{lem38}) for any $f\in L^{\infty}(\Omega)$, since $L^{\infty}(\Omega)$ is dense in $L^{\Phi}(\Omega)$. So let $f \in L^{\infty}(\Omega)$ and $x \in \mathbb{R}^{N}$. Let  $g \in L^{\widetilde{\Phi}}(\Omega)$ such that $\|g\|_{\widetilde{\Phi},\Omega}\leq 1$. By the fact that the measure $\mu$ is invariant under the dynamical system $T$ and the Fubini's theorem, we have 
		\begin{equation*}
			\begin{array}{rcl}
				& & \left|\displaystyle \int_{\Omega} [f(T(x)\omega) - f(\omega) ]\cdot g(\omega) d\mu \right| \\
				& \leq&  \displaystyle  \int_{\Omega} |f(T(x)\omega) - f(\omega) |\cdot |g(\omega)| d\mu  \\
				& = & \frac{1}{|B_{N}|} \displaystyle\int_{\Omega}\int_{B_{N}} |f(T(x+y)\omega) - f(T(y)\omega) |\cdot |g(T(y)\omega)| dyd\mu, 
			\end{array}
		\end{equation*}	
		where $|B_{N}|$ denotes the volume of the unit ball $B_{N}$ in $\mathbb{R}^{N}$. Now we recall that the translations form a strongly continuous group of operators in $L^{1}_{loc}(\mathbb{R}^{N})$. Therefore, the Lebesgue dominated convergence theorem implies
		that $\left| \int_{\Omega} [f(T(x)\omega) - f(\omega) ]\cdot g(\omega) d\mu \right| \to 0$, as $x\to 0$.
		Hence
		\begin{equation*}
			\begin{array}{rcl}
				& &	\| f\circ T(x) - f \|_{(\Phi),\Omega} \\
				& =&  \sup \left\{ \left| \displaystyle \int_{\Omega} [f(T(x)\omega) - f(\omega) ]\cdot g(\omega) d\mu \right| \, ; \, g \in L^{\widetilde{\Phi}}(\Omega) \, \textup{and}\, \|g\|_{\widetilde{\Phi},\Omega}\leq 1  \right\} \\
				& \longrightarrow & 0,
			\end{array}
		\end{equation*}
		as $x\to 0$.
	\end{proof} 
	Now, we consider  based on (\ref{group1S})  the $N$ 1-parameter group of isometries $\{ U(t\vec{e_{i}}) : L^{\Phi}(\Omega)\rightarrow L^{\Phi}(\Omega)\; ; \; t \in \mathbb{R} \}$ where, for $i=1, \cdots, N$, $\vec{e_{i}}= (\delta_{ij})_{1\leq j\leq N}$, $\delta_{ij}$ being the Kronecker symbol. We denote by $D_{i,\Phi}$ the generator (see, e.g. \cite[Page 376]{rudin}) of $U(t\vec{e_{i}})$ and by $\textbf{D}_{i,\Phi}$ its domain. Thus, for $f \in L^{\Phi}(\Omega)$, $f$ is in $\mathbf{D}_{i,\Phi}$ if and only if the limit $D_{i,\Phi}f$ defined by 
	\begin{equation*}
		D_{i,\Phi}f = \lim_{t \to 0} \dfrac{U(t\vec{e_{i}})f - f}{t}
	\end{equation*}  
	exists strongly in $L^{\Phi}(\Omega)$, i.e. $\displaystyle \lim_{t \to 0} \frac{1}{t}\left\|U(t\vec{e_{i}})f - f \right\|_{L^{\Phi}(\Omega)} = 0$. 
	In this case, $D_{i,\Phi}f$ is called $i$-th stochastic derivative of $f$.
	\begin{remark}
		If $\Omega = Y \subset \mathbb{R}^{N}$, $d\mu = dy$ and that we consider the $N$-dimensional dynamical system of translation $\{ T(x) : Y\rightarrow Y;  \; x \in \mathbb{R}^{N} \}$ defined by 
		\begin{equation*}
			T(x)y = (x+y)\textup{mod\,1},
		\end{equation*}
		then the $i$-th stochastic derivative $D_{i,\Phi}f$ of an element $f\in L^{\Phi}(Y)$ is also the $i$-th partial derivative of $f$.
	\end{remark}
	
	For a multi-index $\alpha= (\alpha_{1}, \cdots, \alpha_{N}) \in \mathbb{N}^{N}$, one can naturally define higher order stochastic derivatives by setting:
	\begin{equation*}
		D_{\Phi}^{\alpha} = D_{1,\Phi}^{\alpha_{1}}\circ \cdots\circ D_{N,\Phi}^{\alpha_{N}} \;\;\; \textup{and} \;\;\;	D_{i,\Phi}^{\alpha_{i}} = 
		\underset{\alpha_{i}-times}{\underbrace{D_{i,\Phi}\circ \cdots\circ D_{i,\Phi}}}, \quad i=1, \cdots, N.
	\end{equation*}
	Now we need to define the stochastic analog of the smooth functions on $\mathbb{R}^{N}$.\\
	We set $\displaystyle\textbf{D}_{\Phi}(\Omega) = \bigcap_{i=1}^{N} \textbf{D}_{i,\Phi}(\Omega)$, and define 
	\begin{equation*}
		\textbf{D}_{\Phi}^{\infty}(\Omega) = \left\{ f \in L^{\Phi}(\Omega) \; : \; D^{\alpha}_{\Phi}f \in \textbf{D}_{\Phi}(\Omega), \; \forall \alpha \in \mathbb{N}^{N}  \right\}.
	\end{equation*}
	We recall that (see \cite[Page 99]{sango}), for the Lebesgue spaces $L^{p}(\Omega)$, with $1\leq p \leq \infty$, we have 
	$\displaystyle\textbf{D}_{p}(\Omega) = \bigcap_{i=1}^{N} \textbf{D}_{i,p}$ and 
	\begin{equation*}
		\textbf{D}_{p}^{\infty}(\Omega) = \left\{ f \in L^{p}(\Omega) \; : \; D^{\alpha}_{p}f \in \textbf{D}_{p}(\Omega), \; \forall \alpha \in \mathbb{N}^{N}  \right\},
	\end{equation*}
	where $D^{\alpha}_{p}$ stands for stochastic higher order derivative in $L^{p}(\Omega)$. 
	Furthermore, for $p=\infty$ it is a fact that each element of $\textbf{D}_{\infty}^{\infty}(\Omega)$ possesses stochastic derivatives of any order that are bounded. It is noted by the suggestive symbol $\mathcal{C}^{\infty}(\Omega)$ and is endowed with its natural topology defined by the family of seminorms 
	\begin{equation*}
		N_{n}(f) = \sup_{|\alpha|\leq n} \sup_{\omega\in\Omega} |D_{\infty}^{\alpha}f(\omega)|, \;\; \textup{where} \;\, f\in \mathcal{C}^{\infty}(\Omega) \;\, \textup{and} \;\, |\alpha|= \alpha_{1}+ \cdots + \alpha_{N}. 
	\end{equation*}
		Note that the space $\mathcal{C}^{\infty}(\Omega)$ is a generalization to probability sets of the usual space $\mathcal{C}^{\infty}(\mathbb{R}^{N})$ of smooth functions.  
	Let now $K \in \mathcal{C}^{\infty}_{0}(\mathbb{R}^{N})$ be a nonnegative even function such that $\int_{\mathbb{R}^{N}} K(x)dx = 1$. \\ We set $K_{\delta}(x)= \delta^{-N}K(x/\delta)$, with $\delta >0$ and then define the operator $J_{\delta}$ of $L^{\Phi}(\Omega)$ into $L^{\Phi}(\Omega)$ by 
	\begin{equation*}
		J_{\delta}f = \int_{\mathbb{R}^{N}} K_{\delta}(y) U(y)f \, dy. 
	\end{equation*}  
	As in \cite[page 138]{pankov}, for all $f \in L^{\Phi}(\Omega)$, the function $J_{\delta}f \in \mathcal{C}^{\infty}(\Omega)$ and we have the following lemma.
	\begin{lemma}\label{rem3}
		For any $f \in L^{\Phi}(\Omega)$, we have 
		\begin{equation}\label{rem6}
			\lim_{\delta \to 0} \|J_{\delta}f - f\|_{L^{\Phi}(\Omega)} = 0.
		\end{equation}
	\end{lemma}
	\begin{proof}
		Let $g \in L^{\widetilde{\Phi}}(\Omega)$. Then, 
		\begin{equation*}
			\begin{array}{rcl}
				& &	\left|\displaystyle \int_{\Omega} (J_{\delta}f(\omega) - f(\omega))g(\omega) d\mu  \right| \\
				& \leq & \displaystyle \int_{\mathbb{R}^{N}} K_{\delta}(y) \left( \int_{\Omega} |U(y)f(\omega) - f(\omega) |\cdot |g(\omega)| d\mu  \right) dy \\
				& = & \displaystyle \int_{\mathbb{R}^{N}} K(y) \left( \int_{\Omega} |U(\delta y)f(\omega) - f(\omega) |\cdot |g(\omega)| d\mu  \right) dy \\
				& \leq & 2 \|g\|_{L^{\widetilde{\Phi}}(\Omega)} \displaystyle \int_{\varSigma} K(y) \|U(\delta y)f - f \|_{L^{\Phi}(\Omega)} dy  \\
				& \leq & 2 \|g\|_{L^{\widetilde{\Phi}}(\Omega)} \sup_{y\in \varSigma} \|U(\delta y)f - f \|_{L^{\Phi}(\Omega)}, 
			\end{array}
		\end{equation*}
		where $\varSigma = \textup{supp} K$.  Thus,
		\begin{equation}\label{rem5}
			\begin{array}{rcl}
				\|J_{\delta}f - f\|_{L^{\Phi}(\Omega)} & = & \sup_{\|g\|_{\widetilde{\Phi},\Omega} \leq 1} \left|\displaystyle \int_{\Omega} (J_{\delta}f(\omega) - f(\omega))g(\omega) d\mu  \right| \\
				& \leq & 2 \sup_{y\in \varSigma} \|U(\delta y)f - f \|_{L^{\Phi}(\Omega)}.
			\end{array}
		\end{equation}
		Since $\varSigma$ is compact and the group $U(x)$ for $x \in \mathbb{R}^{N}$ is strongly continuous in $L^{\Phi}(\Omega)$, then (\ref{rem5}) implies (\ref{rem6}) when $\delta \to 0$. 
	\end{proof}
	
	\noindent From the Lemma \ref{rem3}, we deduce the following proposition.
	\begin{proposition}
		Let $\Phi \in \Delta_{2}$ be a Young function. The space $ \mathcal{C}^{\infty}(\Omega)$ is dense in $L^{\Phi}(\Omega)$.
	\end{proposition}
	By a stochastic distribution on $\Omega$ is meant any continuous linear mapping from $\mathcal{C}^{\infty}(\Omega)$ to the real field $\mathbb{R}$. We denote the space of stochastic distributions by $(\mathcal{C}^{\infty}(\Omega))'$.
	\\
	For any $\alpha \in \mathbb{N}^{N}$, we define the stochastic weak derivative of $f \in (\mathcal{C}^{\infty}(\Omega))'$ as follows:
	\begin{equation*}
		(D^{\alpha}f)(\phi) = (-1)^{|\alpha|}f(D^{\alpha}\phi), \quad \forall \phi \in \mathcal{C}^{\infty}(\Omega).
	\end{equation*}
	Let $\Phi \in \Delta_{2}$ be a Young function and $\widetilde{\Phi}$ its conjugate.
	As $\mathcal{C}^{\infty}(\Omega)$ is dense in $L^{\widetilde{\Phi}}(\Omega)$, it is immediate that $(L^{\widetilde{\Phi}}(\Omega))' = L^{\Phi}(\Omega) \subset (\mathcal{C}^{\infty}(\Omega))'$, so that one may define the stochastic weak derivative of any $f \in L^{\Phi}(\Omega)$, and it verifies the functional equation: 
	\begin{equation}\label{deri1}
		(D^{\alpha}f)(\phi) = (-1)^{|\alpha|} \int_{\Omega} f\, D^{\alpha}\phi\,d\mu, \quad \forall \phi \in \mathcal{C}^{\infty}(\Omega).
	\end{equation}	
	In particular for $f \in \textbf{D}_{i,\Phi}$, we have
	\begin{equation}\label{deri2}
		\int_{\Omega}\phi D_{i,\Phi}f \,d\mu = - \int_{\Omega} f\, D_{i,\infty}\phi\,d\mu, \quad \forall \phi \in \mathcal{C}^{\infty}(\Omega).
	\end{equation}
	So that we may identify $ D_{i,\Phi}$ with $D^{\alpha_{i}}$, where $\alpha_{i} = (\delta_{ij})_{1\leq j\leq N}$.  
		Note that, formulas (\ref{deri1}) and (\ref{deri2}) are analogous to Sango-Woukeng paper \cite{sango} and the same notation as them has been adopted in the entire subsection. 
	\\
	Conversely, if $f \in L^{\Phi}(\Omega)$ is such that there exists $f_{i} \in L^{\Phi}(\Omega)$ with \\ $(D^{\alpha_{i}}f)(\phi) = - \int_{\Omega} f_{i} \phi\ d\mu$ for all $\phi \in \mathcal{C}^{\infty}(\Omega)$, then $f \in \textbf{D}_{i,\Phi}$ and $D_{i,\Phi}f = f_{i}$. 
	\\
	Endowing $\displaystyle\textbf{D}_{\Phi}(\Omega) = \bigcap_{i=1}^{N} \textbf{D}_{i,\Phi}(\Omega)$ with the natural graph norm 
	\begin{equation*}
		\|f\|_{\textbf{D}_{\Phi}(\Omega)} = \|f\|_{L^{\Phi}(\Omega)} + \sum_{i=1}^{N} \|D_{i,\Phi}f\|_{L^{\Phi}(\Omega)}, \quad \forall f \in \textbf{D}_{\Phi}(\Omega),
	\end{equation*}
	we obtain a Banach space representing the stochastic generalization of the classical Orlicz-Sobolev space $W^{1}L^{\Phi}(\mathbb{R}^{N})$, and so, we denote it by $W^{1}L^{\Phi}(\Omega)$.  \\
	%\begin{defi}\cite{sango}(invariance function and ergodic dynamical system) \\
	Now we recall some definitions about the invariance of functions and ergodic dynamical system.  Note that these definitions are a generalization to Orlicz-Sobolev spaces $W^{1}L^{\Phi}$ of the definitions obtained in \cite[Page 99]{sango} for the case of Lebesgue-Sobolev spaces $W^{1,p}$. \label{pageS} \\  	A function $f \in L^{\Phi}(\Omega)$ is said to be invariant for the dynamical system $T$ (relative to $\mu$) if for any $x \in \mathbb{R}^{N}$
	\begin{equation*}
		f\circ T(x) = f, \quad \mu-a.e. \;\, \textup{on} \;\, \Omega.
	\end{equation*}
	We denote  $I_{nv}^{\Phi}(\Omega)$ the set of functions in $L^{\Phi}(\Omega)$ that are invariant for $T$. 
	The set $I_{nv}^{\Phi}(\Omega)$ is a closed vector subspace of $L^{\Phi}(\Omega)$   
		and when $\Phi(t) = \frac{t^{p}}{p}$ ($p\geq 1$), then $I_{nv}^{\Phi}(\Omega)$ becomes the space $I_{nv}^{p}(\Omega)$ obtained by replacing the Orlicz space $ L^{\Phi}(\Omega)$ in the definition with the Lebesgue space  $L^{p}(\Omega)$. 
	The dynamical system $T$ is said to be ergodic if every invariant function is $\mu$-equivalent to a constant, i.e. for all $f \in I_{nv}^{\Phi}(\Omega)$, we have $f(\omega) = \textup{c}$, $c\in \mathbb{R}$, for $\mu$-a.e. $\omega\in \Omega$.  
	Let $f$ be a measurable function in $\Omega$, for a fixed $\omega\in \Omega$ the function $x\rightarrow f(T(x)\omega)$, $x\in \mathbb{R}^{N}$ is called a realization of $f$ and the mapping $(x,\omega) \rightarrow f(T(x)\omega)$ is called a stationary process. The process is said to be stationary ergodic if the dynamical system $T$ is ergodic. \\
	%\end{defi}
	Thus,	for $f \in \textbf{D}_{1}^{\infty}(\Omega)$, and for $\mu$-a.e. $\omega \in \Omega$, the function $x \rightarrow f(T(x)\omega)$ is in $\mathcal{C}^{\infty}(\mathbb{R}^{N})$ and further 
	\begin{equation}\label{lem18}
		D^{\alpha}_{x} f(T(x)\omega) = (D^{\alpha}_{1}f)(T(x)\omega), \quad \textup{for\; any} \;\, \alpha \in \mathbb{N}^{N},
	\end{equation}
	where $D^{\alpha}_{x}$ is the usual higher order derivative with respect to $x$ and $D^{\alpha}_{1}$ is the stochastic higher order derivative in $L^{1}(\Omega)$.
	If the dynamical system $T$ is ergodic, then :
	\begin{itemize}
		\item[i)] For $f\in L^{1}(\Omega)$, we have $f \in I_{nv}^{1}(\Omega)$ if and only if $D_{i,1}f = 0$ for each $1 \leq i \leq N$.
		\item[ii)] For $f\in L^{\Phi}(\Omega)$, we have $f \in I_{nv}^{\Phi}(\Omega)$ if and only if $D_{i,\Phi}f = 0$ for each $1 \leq i \leq N$.
	\end{itemize}
	So if we endow $\mathcal{C}^{\infty}(\Omega)$ with the seminorm 
	\begin{equation*}
		\|u\|_{\# ,\Phi} = \sum_{i=1}^{N} \|D_{i,\Phi}u\|_{L^{\Phi}(\Omega)}, \quad u \in \mathcal{C}^{\infty}(\Omega),
	\end{equation*}
	we obtain a locally convex space which is generally not complete. We denote $W^{1}_{\#}L^{\Phi}(\Omega)$ the completion of $\mathcal{C}^{\infty}(\Omega)$ with respect to the seminorm $\|\cdot\|_{\# ,\Phi}$, and we denote by 
	\begin{equation*}
		I_{\Phi} : \mathcal{C}^{\infty}(\Omega) \hookrightarrow W^{1}_{\#}L^{\Phi}(\Omega),
	\end{equation*}
	the canonical mapping.\\
	The gradient operator $D_{\omega,\Phi} = (D_{1,\Phi}, \cdots, D_{N,\Phi}) : \mathcal{C}^{\infty}(\Omega) \rightarrow L^{\Phi}(\Omega)^{N}$ extends by continuity to a unique mapping $\overline{D}_{\omega,\Phi} = (\overline{D}_{1,\Phi}, \cdots, \overline{D}_{N,\Phi}) : W^{1}_{\#}L^{\Phi}(\Omega) \rightarrow L^{\Phi}(\Omega)^{N}$ with the properties :
	\begin{itemize}
		\item[i)] $D_{i,\Phi} = \overline{D}_{i,\Phi}\circ I_{\Phi}$, for $1\leq i\leq N$,
		\item[ii)] $\displaystyle \|u\|_{W^{1}_{\#}L^{\Phi}(\Omega)}= \|u\|_{\# ,\Phi} = \sum_{i=1}^{N} \|\overline{D}_{i,\Phi}u\|_{L^{\Phi}(\Omega)}, \quad u \in W^{1}_{\#}L^{\Phi}(\Omega)$.
	\end{itemize}
	Moreover, the mapping $\overline{D}_{\omega,\Phi}$ is an isometric embedding of $W^{1}_{\#}L^{\Phi}(\Omega)$ into a closed subspace of $L^{\Phi}(\Omega)^{N}$ and as $L^{\Phi}(\Omega)^{N}$ is a reflexive Banach space, then we deduce that the Banach space $W^{1}_{\#}L^{\Phi}(\Omega)$ is reflexive.  \\
	By duality we define the operator $\textup{div}_{\omega,\widetilde{\Phi}} : L^{\widetilde{\Phi}}(\Omega)^{N} \rightarrow (W^{1}_{\#}L^{\Phi}(\Omega))'$ by :
	\begin{equation*}
		\langle \textup{div}_{\omega,\widetilde{\Phi}} u \, , \, v \rangle = - \langle u\, , \, \overline{D}_{\omega,\Phi}v \rangle, \quad \forall u \in L^{\widetilde{\Phi}}(\Omega)^{N} \; \textup{and} \; \forall v \in W^{1}_{\#}L^{\Phi}(\Omega).
	\end{equation*}  
	The operator $\textup{div}_{\omega,\widetilde{\Phi}}$ just defined extends the natural divergence operator defined in $\mathcal{C}^{\infty}(\Omega)$ since for all $f\in \mathcal{C}^{\infty}(\Omega)$ we have $D_{i,\Phi}f = \overline{D}_{i,\Phi}\circ I_{\Phi}(f)$. 
		We will denote $D_{\omega,\Phi}=D_{\omega}$, $\overline{D}_{\omega,\Phi}=\overline{D}_{\omega}$ and $\textup{div}_{\omega,\widetilde{\Phi}}= \textup{div}_{\omega,p'}$ ($p'= \frac{p-1}{p}$) if there is no ambiguity, and we refer to \cite{sango} for the definitions of these operators in $L^{p}$-spaces that the authors note $D_{\omega,p}$, $\overline{D}_{\omega,p}$ and $\textup{div}_{\omega,p'}$ respectively.  \\
	To end this section, we define the Orlicz-Sobolev space $W^{1}L^{\Phi}_{D_{x}}(Q; L^{\Phi}(\Omega))$ as follows
	\begin{equation}\label{c1orli1S}
		W^{1}L^{\Phi}_{D_{x}}(Q; L^{\Phi}(\Omega)) = \left\{  v \in L^{\Phi}(Q\times\Omega) \, : \, \frac{\partial v}{\partial x_{i}} \in L^{\Phi}(Q\times\Omega), \; 1 \leq i \leq N \right\}
	\end{equation}
	where derivatives are taken in the distributional sense on $Q$. Endowed with the norm 
	\begin{equation*}
		\lVert v \rVert_{W^{1}L^{\Phi}_{D_{x}}(Q; L^{\Phi}(\Omega))} = \lVert v \rVert_{L^{\Phi}(Q\times\Omega)} + \sum_{i=1}^{n} \lVert D_{x_{i}} v \rVert_{L^{\Phi}(Q\times\Omega)}\, ; \; v \in W^{1}L^{\Phi}_{D_{x}}(Q; L^{\Phi}(\Omega)),
	\end{equation*}
	$W^{1}L^{\Phi}_{D_{x}}(Q; L^{\Phi}(\Omega))$ is a reflexive Banach space. On the other hand, we denote by $W^{1}_{0}L^{\Phi}_{D_{x}}(Q; L^{\Phi}(\Omega))$ the set of functions in $W^{1}L^{\Phi}_{D_{x}}(Q; L^{\Phi}(\Omega))$ with zero boundary condition on $Q$. Endowed with the norm 
	\begin{equation*}
		\lVert v \rVert_{W^{1}_{0}L^{\Phi}_{D_{x}}(Q; L^{\Phi}(\Omega))} =  \sum_{i=1}^{n} \lVert D_{x_{i}} v \rVert_{L^{\Phi}(Q\times\Omega)} \;\, ; \quad v \in W^{1}_{0}L^{\Phi}_{D_{x}}(Q; L^{\Phi}(\Omega)),
	\end{equation*}
	$W^{1}_{0}L^{\Phi}_{D_{x}}(Q; L^{\Phi}(\Omega))$ is a reflexive Banach space.

	%%%%%%%%%%%%%%%%%%%%%%%%%%%%%%%%%%%%%%%%%%%%%%%%%%%%%%%%%%%%%%%%%%%%%%%%%%%%%%%%%%%%%
	%%%%%%%%%%%%%%%%%%%%    3eme PARTIE   %%%%%%%%%%%%%%%%%%%%%%%%%%%%%%%%%%%%%%%%%%%%%%%
	
	\section{Stochastic two-scale convergence in the mean in Orlicz-Sobolev's spaces}\label{sect3}
	
	The stochastic homogenization of functional $F_{\epsilon}$ define in (\ref{ras1}), amounts to find a homogenized functional $F$, by using the stochastic two-scale convergence in the mean method (see \cite{bourgeat}), such that the sequence of minimizers $u_{\epsilon}$ converges to a limit $u$, which is precisely the minimizer of $F$. For this purpose, we need to extend the notion of stochastic two-scale convergence in the mean in $L^{\Phi}(Q\times\Omega)$.  
	
	\subsection{Fundamentals of stochastic homogenization}
	
	Throughout the paper the letter $E$ will denote any ordinary sequence $E=(\epsilon_{n})$ (integers $n\geq 0$) with $0 \leq \epsilon_{n} \leq 1$ and $\epsilon_{n} \rightarrow 0$ as $n\rightarrow \infty$. Such a sequence will be termed a \textit{fundamental sequence}.
	We now let $\{ T(x), \, x \in \mathbb{R}^{N} \}$ be a fixed $N$-dimensional dynamical system on $\Omega$, with respect to a fixed probability measure $\mu$ on $\Omega$ invariant for $T$. We also assume that $L^{\Phi}(\Omega)$ is separable and reflexive.
	Let $Q$ be a bounded domain in $\mathbb{R}^{N}$. We start by defining the important notion of admissible function. 
	\begin{definition}
		We say that an element $f \in L^{\Phi}(Q\times\Omega)$ is admissible if the function 
		\begin{equation*}
			f_{T} : (x, \omega) \rightarrow f(x, T(x)\omega), \quad (x,\omega) \in Q\times \Omega,
		\end{equation*}
		defines an element of $L^{\Phi}(Q\times\Omega)$.
	\end{definition}
	Not every element of $L^{\Phi}(Q\times\Omega)$ is admissible ($f_{T}$ can fail to be measurable for certain $f \in L^{\Phi}(Q\times\Omega)$). In section \ref{sect4}, the  integral functional to which we apply the results of this section will have an admissible integrand. It is thus important to exhibit a class of admissible functions which is large enough to encompass significant applications, but whose description is simple enough to make verification tractable in practical situations.  As in Bourgeat \textit{et al.} \cite{bourgeat}, we will discuss three such classes of functions in what follows. 
	\begin{proposition}
		Assume that $\Phi \in \Delta'$. Then
		every linear combination of functions of the form 
		\begin{equation*}
			(x, \omega) \rightarrow u(x)v(\omega), \quad (x, \omega) \in Q\times\Omega,\;\; u \in L^{\Phi}(Q),\, v \in L^{\Phi}(\Omega),
		\end{equation*}
		is admissible.	
	\end{proposition}
	\begin{proof}
		Let $u, v$ such that there exist $\lambda > 0$, $\nu > 0$ and $\int_{Q} \Phi\left(\dfrac{|u(x)|}{\lambda}\right)dx < \infty$, $\int_{\Omega} \Phi\left(\dfrac{|v(\omega)|}{\nu}\right)d\mu < \infty$. Then the function $(x, \omega)\rightarrow f(x, \omega) = u(x)v(\omega)$ is measurable on $Q\times\Omega$ and since  $\Phi \in \Delta'$, there exists $\beta > 0$ such that 
		\begin{equation*}
			\begin{array}{rcl}
				\displaystyle	\iint_{Q\times \Omega} \Phi\left(\dfrac{|f(x,\omega)|}{\lambda \nu}\right)  dxd\mu &=& \displaystyle\iint_{Q\times \Omega} \Phi\left(\dfrac{|u(x)|\cdot|v(\omega)|}{\lambda \nu}\right)  dxd\mu \\
				& \leq&  \beta \, \displaystyle \int_{Q} \Phi\left(\dfrac{|u(x)|}{\lambda}\right)dx \cdot \int_{\Omega} \Phi\left(\dfrac{|v(\omega)|}{\nu}\right)d\mu < \infty,
			\end{array}
		\end{equation*} 
		thus $f \in L^{\Phi}(Q\times\Omega)$.  \\
		Now, since the function $(x, \omega)\rightarrow v(T(x)\omega)$ is measurable on $Q\times\Omega$, then it is also true for the function $(x, \omega)\rightarrow f(x, T(x)\omega) = u(x)v(T(x)\omega)$ and by the fact that the measure $\mu$ is invariant under the dynamical system $T$, we have 
		\begin{equation*}
			\begin{array}{rcl}
				\displaystyle	\iint_{Q\times \Omega} \Phi\left(\dfrac{|f_{T}(x,\omega)|}{\lambda \nu}\right)  dxd\mu  &\leq & \beta \;\displaystyle \int_{Q} \Phi\left(\dfrac{|u(x)|}{\lambda}\right)dx \cdot \int_{\Omega} \Phi\left(\dfrac{|v(T(x)\omega)|}{\nu}\right)d\mu \\
				&= & \beta \;\displaystyle \int_{Q} \Phi\left(\dfrac{|u(x)|}{\lambda}\right)dx \cdot \int_{\Omega} \Phi\left(\dfrac{|v(\omega)|}{\nu}\right)d\mu_{T} \\
				& =& \beta \;\displaystyle \int_{Q} \Phi\left(\dfrac{|u(x)|}{\lambda}\right)dx \cdot \int_{\Omega} \Phi\left(\dfrac{|v(\omega)|}{\nu}\right)d\mu \\
				& < & \infty
			\end{array}
		\end{equation*}
		thus $f_{T} \in L^{\Phi}(Q\times\Omega)$.
	\end{proof}
	\begin{proposition}
		Let $B(\Omega)$ denote the Banach space of all functions defined everywhere on $\Omega$ that are bounded and measurable over $\Omega$. Let $L^{\Phi}(Q, B(\Omega))$ denote the set of all elements $f \in L^{\Phi}(Q\times\Omega)$  such that for a.e. $x \in Q$, $f(x,\cdot) \in B(\Omega)$, and there exists a $\lambda > 0$ such that
		\begin{equation}\label{c3boch1}
			\int_{Q} \Phi\left(\dfrac{\|f(x, \cdot)\|_{B(\Omega)}}{\lambda}\right)  dx < \infty.
		\end{equation}  
		Then every element $f$ of $L^{\Phi}(Q, B(\Omega))$ is admissible.
	\end{proposition}
	\begin{proof}
		As in the case of standard Sobolev spaces (see \cite[Proposition 3.1]{bourgeat}), the proof is done in two steps. The first step is based on the decomposition into simple elements of an Orlicz-Bochner type function (i.e., of form (\ref{c3boch1})) to obtain the measurability of $f$, and the second step is based on the $\mu$-invariance of the dynamical system $T$ to conclude on the admissibility of $f$. So let's start with the first step.  \\
			\textbf{Step 1.} Since  $f \in L^{\Phi}(Q, B(\Omega))$  is Bochner $\Phi$-integrable relative to $B(\Omega)$, then by \cite[Theorem II.2.2]{jr},  there exists a sequence $f_{n} = \sum_{i} f_{in} \chi_{A_{in}}$ of $B(\Omega)$-valued simple functions on $Q$ such that 
			\begin{equation*}
				\lim_{n} \int_{Q} \Phi\left(\|f(x,\cdot) - f_{n}(x,\cdot)\|_{B(\Omega)} \right) dx = 0,
			\end{equation*}  
			and by standard argument, we deduce from this the existence of a subsequence, still denoted $(f_{n})$ such that 
			\begin{equation*}\label{vrai1}
				\lim_{n} \|f(x,\cdot) - f_{n}(x,\cdot)\|_{B(\Omega)} = 0 \quad \textup{for\, a.e.} \; x \in Q.
			\end{equation*}
			The remainder of this first step to conclude that $f$ is measurable on $Q\times \Omega$ is entirely identical to that in \cite[Page 31]{bourgeat} and is therefore omitted here. \\
			\textbf{Step 2.} Let  $\lambda > 0$. Since $f \in L^{\Phi}(Q, B(\Omega))$, then by the $\mu$-invariance of dynamical system $T$, we have 
			\begin{equation*}
				\int_{Q\times\Omega} \Phi\left(\dfrac{|f(x,T(x)\omega)|}{\lambda}\right) dx d\mu = \int_{Q\times\Omega} \Phi\left(\dfrac{|f(x,\omega)|}{\lambda}\right) dx d\mu,
			\end{equation*} 
			and therefore, we get that $f_{T} \in L^{\Phi}(Q\times\Omega)$. 
	\end{proof}
		The last class of admissible functions is given by the following proposition whose  proof is entirely identical to the proof of \cite[Proposition 3.2]{bourgeat}, since no argument changes when replacing the classical Lebesgue space $L^{2}$ by Orlicz space $L^{\Phi}$.  
	\begin{proposition}
		Every element $f \in \mathcal{C}(\overline{Q}, L^{\infty}(\Omega))$ is admissible.
	\end{proposition}	
	\noindent We assume for the rest of the work that $\Phi$ and $\widetilde{\Phi}$ satisfy the $\Delta'$-condition. We can now define the stochastic two-scale convergence in the mean in the Orlicz space.
	\begin{definition}(stochastic 2-scale convergence in the mean in $L^{\Phi}$-spaces)\label{lem36} \\
		Let $\Phi \in \Delta_{2}$ be a Young function and $\widetilde{\Phi}$ its conjugate. A sequence $(u_{\epsilon})$ in $ L^{\Phi}(Q\times\Omega)$ is said to weakly stochastically 2-scale converge in the mean to $u_{0} \in  L^{\Phi}(Q\times\Omega)$, if as $\epsilon \to 0$ we have,  
		\begin{equation}\label{lem35}
			\lim_{\epsilon \to 0} \iint_{Q\times \Omega} u_{\epsilon}(x,\omega) f(x, T(\epsilon^{-1}x)\omega) dx\, d\mu = \iint_{Q\times \Omega} u_{0}(x,\omega) f(x,\omega) dx\, d\mu
		\end{equation}
		for every $f \in \mathcal{C}^{\infty}_{0}(Q)\otimes\mathcal{C}^{\infty}(\Omega)$. \\
		We express this by writing  $u_{\epsilon} \;\rightarrow\; u_{0}$ stoch. mean in $ L^{\Phi}(Q\times\Omega)$-weak 2s.
	\end{definition}
	\noindent We recall that $\mathcal{C}^{\infty}_{0}(Q)\otimes\mathcal{C}^{\infty}(\Omega)$ is the space of functions of the form, 
	\begin{equation*}
		f(x, \omega) = \sum_{finite} \varphi_{i}(x)\psi_{i}(\omega), \quad (x,\omega) \in Q\times \Omega,
	\end{equation*}
	with $\varphi_{i} \in \mathcal{C}^{\infty}_{0}(Q)$, $\psi_{i} \in \mathcal{C}^{\infty}(\Omega)$. Such functions are dense in $\mathcal{C}^{\infty}_{0}(Q) \otimes L^{\widetilde{\Phi}}(\Omega)$, since $\mathcal{C}^{\infty}(\Omega)$ is dense in $L^{\widetilde{\Phi}}(\Omega)$ and hence in $\mathcal{K}(Q; L^{\widetilde{\Phi}}(\Omega))$,  $\mathcal{K}(Q; L^{\widetilde{\Phi}}(\Omega))$ being the space of continuous functions of $Q$ into $L^{\widetilde{\Phi}}(\Omega)$ with compact support contained in $Q$. As $\mathcal{K}(Q; L^{\widetilde{\Phi}}(\Omega))$ is dense in $L^{\widetilde{\Phi}}(Q\times\Omega)$, the uniqueness of the stochastic two-scale limit is ensured.
	
	\begin{remark}(Inspired by \cite{sango})\label{lem16} 
		\begin{enumerate}
			\item[$R_{1})$]\textit{ Assume that a sequence $(u_{\epsilon})$ in $ L^{\Phi}(Q\times\Omega)$ is weakly stochastically 2-scale convergent in the mean to $u_{0} \in  L^{\Phi}(Q\times\Omega)$. Then since $\mathcal{C}^{\infty}_{0}(Q)\otimes\mathcal{C}^{\infty}(\Omega)$ is dense in $\mathcal{C}(Q; L^{\infty}(\Omega))$, then (\ref{lem35}) holds for every $f \in \mathcal{C}(Q; L^{\infty}(\Omega))$. } 
			\item[$R_{2})$] \textit{ In the Definition \ref{lem36},  if we take $\Omega = Y \subset \mathbb{R}^{N}$, $d\mu = dy$, and that we consider the $N$-dimensional dynamical system of translation $\{ T(x) : Y\rightarrow Y \; ; \; x \in \mathbb{R}^{N} \}$ defined by $T(x)y = (x+y)\textup{mod1}$, we obtain the framework to which periodic two-scale convergence is defined as in \cite{tacha1}. } 
		\end{enumerate}
	\end{remark}

	\subsection{Compactness theorem: case of Orlicz's spaces}
	The proof of the compactness theorem required the following lemma.
	\begin{lemma}\label{lem2}
		Let	$f \in \mathcal{C}^{\infty}_{0}(Q)\otimes\mathcal{C}^{\infty}(\Omega)$.
		Then, for any $\epsilon >0$, the function $f_{T}^{\epsilon} : (x, \omega) \rightarrow f(x, T(\epsilon^{-1}x)\omega)$ which is in $L^{\Phi}(Q\times\Omega)$, is such that 
		\begin{equation*}\label{lem1}
			\|f_{T}^{\epsilon}\|_{L^{\Phi}(Q\times \Omega)}  =  \|f\|_{L^{\Phi}(Q\times \Omega)}.
		\end{equation*}
	\end{lemma}
	\begin{proof}
		For that, it is sufficient to do it for $f$ under the form $f(x, \omega) = \phi(x)\psi(\omega)$ with $\phi \in \mathcal{C}^{\infty}_{0}(Q)$ and $\psi \in \mathcal{C}^{\infty}(\Omega)$. \\
		Considering the fact that the measure $\mu$ is invariant under the dynamical system $T$, we have for any $\epsilon >0$, 
		\begin{equation*}
			\begin{array}{rcl}
				\| f_{T}^{\epsilon} \|_{\Phi, Q\times\Omega}  &= & \displaystyle \inf \left\{k >0 \; ; \;\, \displaystyle\iint_{Q\times \Omega} \Phi\left(\dfrac{\phi(x)\psi(T(\epsilon^{-1}x)\omega) }{k} \right)dxd\mu \leq 1 \right\} \\
				&= &\displaystyle \inf \left\{k >0 \; ; \;\, \displaystyle \iint_{Q\times \Omega} \Phi\left(\dfrac{\phi(x)\psi(\omega) }{k} \right)dxd\mu_{T} \leq 1 \right\} \\
				& =& \displaystyle\inf \left\{k >0 \; ; \;\, \displaystyle \iint_{Q\times \Omega} \Phi\left(\dfrac{\phi(x)\psi(\omega) }{k} \right)dxd\mu \leq 1 \right\}  \\
				& = & \| f \|_{\Phi, Q\times\Omega}.
			\end{array}
		\end{equation*}
	\end{proof}
	\begin{theorem}(compactness 1)\label{lem4} \\
		Let $\Phi \in \Delta_{2}$ be a Young function.	Any bounded sequence $(u_{\epsilon})_{\epsilon\in E}$  in $L^{\Phi}(Q\times\Omega)$ admits a subsequence which is weakly stochastically 2-scale convergent in the mean in $ L^{\Phi}(Q\times\Omega)$. 
	\end{theorem}
	\begin{proof}
		Let $\Phi \in \Delta_{2}$ be a Young function and $\widetilde{\Phi} \in \Delta_{2}$ its conjugate, $(u_{\epsilon})_{\epsilon\in E}$ a bounded sequence  in $L^{\Phi}(Q\times\Omega)$. \\
			In order to use \cite[Proposition 3.2]{gabri}, we get $Y = L^{\widetilde{\Phi}}(Q\times\Omega)$ and $X = \mathcal{C}^{\infty}_{0}(Q)\otimes\mathcal{C}^{\infty}(\Omega)$.  
		For each $f \in \mathcal{C}^{\infty}_{0}(Q)\otimes\mathcal{C}^{\infty}(\Omega)$ and $\epsilon > 0$, let us define the functional $\Gamma_{\epsilon}$ by 
		\begin{equation*}
			\Gamma_{\epsilon}(f) = \iint_{Q\times \Omega} u_{\epsilon}(x,\omega) f(x, T(\epsilon^{-1}x)\omega) dx\, d\mu.
		\end{equation*}
		Then by Holder's inequality and the Lemma \ref{lem2}, we have
		\begin{equation*}\label{fra1}
			\limsup_{\epsilon} |\Gamma_{\epsilon}(f)| \leq 2c \, \|f\|_{L^{\widetilde{\Phi}}(Q\times\Omega)},
		\end{equation*}
		where $\displaystyle c = \sup_{\epsilon} \|u_{\epsilon}\|_{L^{\Phi}(Q\times\Omega)}$. \\
			Hence, we deduce from \cite[Proposition 3.2]{gabri},  the existence of a subsequence $E'$ of $E$ and of a unique $u_{0} \in Y' = L^{\Phi}(Q\times\Omega)$ such that 
		\begin{equation*}
			\lim_{\epsilon \to 0} \iint_{Q\times \Omega} u_{\epsilon}(x,\omega) f(x, T(\epsilon^{-1}x)\omega) dx\, d\mu = \iint_{Q\times \Omega} u_{0}(x,\omega) f(x,\omega) dx\, d\mu
		\end{equation*}
		for all $f \in \mathcal{C}^{\infty}_{0}(Q)\otimes\mathcal{C}^{\infty}(\Omega)$.
	\end{proof}
	
	\subsection{Compactness theorem: case of Orlicz-Sobolev's spaces}
		Let us begin with the following lemma which is an adaptation to Orlicz setting of \cite[Proposition 1]{sango}. 
	\begin{lemma}\label{lem5}
		Let $\Phi \in \Delta_{2}$ be a Young function and $\widetilde{\Phi}$ its conjugate.	Let $v \in L^{\Phi}(\Omega)^{N}$ satisfy 
		\begin{equation*}
			\int_{\Omega} v\cdot g d\mu = 0, \quad \textup{for \; all} \;\, g \in  \mathcal{V}^{\widetilde{\Phi}}_{\textup{div} }= \{ f \in \mathcal{C}^{\infty}(\Omega)^{N} : \textup{div}_{\omega,\widetilde{\Phi}}f=0 \}.
		\end{equation*}
		Then, there exists $u \in W^{1}_{\#}L^{\Phi}(\Omega)$ such that 
		\begin{equation*}
			v = \overline{D}_{\omega,\Phi}u.
		\end{equation*} 
	\end{lemma}
	\begin{proof}
			As in the proof of \cite[Proposition 1]{sango}, we need to check the following: 
			\begin{enumerate}
				\item[(1)] $\textup{div}_{\omega,\widetilde{\Phi}}$ is closed ;
				\item[(2)]  $(\textup{div}_{\omega,\widetilde{\Phi}})^{*} = - \overline{D}_{\omega,\Phi}$, where $(\textup{div}_{\omega,\widetilde{\Phi}})^{*}$ is the adjoint operator of $(\textup{div}_{\omega,\widetilde{\Phi}})^{*}$ ;
				\item[(3)]  $Range(\overline{D}_{\omega,\Phi})$ is closed in $L^{\Phi}(\Omega)^{N}$ and 
				\item[(4)]  $v$ is orthogonal to the kernel of $\textup{div}_{\omega,\widetilde{\Phi}}$. 
			\end{enumerate}
			However, the proofs of items (1)-(3) are entirely identical and are therefore omitted here.
			For item (4) it suffices to show that $\mathcal{V}^{\widetilde{\Phi}}_{\textup{div}}$ is dense in $\ker(\textup{div}_{\omega,\widetilde{\Phi}})$. To see this, let $g = (g_{i}) \in \ker(\textup{div}_{\omega,\widetilde{\Phi}})$, then by the Lemma \ref{rem3}, for $\delta > 0$, we get $J_{\delta}g \equiv (J_{\delta}g_{i}) \in \mathcal{C}^{\infty}(\Omega)^{N}$ and $J_{\delta}g \rightarrow g$ strongly in $L^{\Phi}(\Omega)^{N}$. Furthermore, for $\delta > 0$, $\textup{div}_{\omega,\widetilde{\Phi}} J_{\delta}g = 0$.   
			This completes the proof.  
	\end{proof}
	\begin{theorem}(Compactness 2)\label{lem29} \\
		Let $\Phi \in \Delta_{2}$ be a Young function and $\widetilde{\Phi} \in \Delta_{2}$ its conjugate. 
		Assume $(u_{\epsilon})_{\epsilon\in E}$ is a sequence in $W^{1}L^{\Phi}_{D_{x}}(Q; L^{\Phi}(\Omega))$ such that: 
		\begin{itemize}
			\item[i)] $(u_{\epsilon})_{\epsilon\in E}$ is bounded in $L^{\Phi}\left(Q\times\Omega\right)$ and $(D_{x}u_{\epsilon})_{\epsilon\in E}$ is bounded in $L^{\Phi}\left(Q\times\Omega\right)^{N}$.
		\end{itemize}
		Then there exist $u_{0} \in W^{1}L^{\Phi}_{D_{x}}(Q; I_{nv}^{\Phi}(\Omega))$, $u_{1} \in L^{1}\left(Q; W^{1}_{\#}L^{\Phi}(\Omega)\right)$ with $\overline{D}_{\omega}u_{1} \in L^{\Phi}(Q\times\Omega)^{N}$  and a subsequence $E'$ from $E$ such that  as $E' \ni \epsilon \rightarrow 0$, 
		\begin{itemize}
			\item[ii)] $u_{\epsilon} \rightarrow u_{0}$ stoch. mean in $L^{\Phi}(Q\times \Omega)$-weak 2s,
			\item[iii)] $D_{x}u_{\epsilon} \rightarrow D_{x}u_{0} + \overline{D}_{\omega}u_{1}$ stoch. mean in $L^{\Phi}(Q\times \Omega)^{N}$-weak 2s, with $u_{1} \in L^{\Phi}(Q; W^{1}_{\#}L^{\Phi}(\Omega))$ when $\widetilde{\Phi} \in \Delta'$.  
		\end{itemize}
	\end{theorem}
	\begin{proof}
			Let us first note that by hypothesis $i)$, Theorem \ref{lem4} implies the existence
			of a subsequence  $E'$ from $E$, a function $u_{0} \in L^{\Phi}(Q\times \Omega)$ and a vector function $\mathbf{v} = (v_{i})_{1\leq i \leq N} \in L^{\Phi}(Q\times \Omega)^{N}$ such that, as $E' \ni \epsilon \rightarrow 0$, we have  $u_{\epsilon} \rightarrow u_{0}$ stoch. mean in $L^{\Phi}(Q\times \Omega)$-weak 2s and 
		\begin{equation*}
			D_{x}u_{\epsilon} \rightarrow \mathbf{v} \;\;\; \textup{stoch.\; mean\; in} \; L^{\Phi}(Q\times \Omega)^{N}-weak\;2s.
		\end{equation*} 
			To complete the proof, we must check that: 
		\begin{enumerate}
			\item[(a)] $u_{0}(x,\cdot) \in I_{nv}^{\Phi}(\Omega)$, that is $D_{\omega}u_{0}(x,\cdot) = 0$ or equivalently  $\int_{\Omega} u_{0}(x,\cdot)D_{i,\Phi}\varphi \, d\mu = 0$ for all $\varphi \in \mathcal{C}^{\infty}(\Omega)$ and $u_{0} \in W^{1}L^{\Phi}(Q; I_{nv}^{\Phi}(\Omega))$ ;
			\item[(b)]  there exists a function $u_{1} \in L^{1}\left(Q; W^{1}_{\#}L^{\Phi}(\Omega)\right)$ with $\overline{D}_{\omega}u_{1} \in L^{\Phi}(Q\times\Omega)^{N}$ such that $\mathbf{v} = D_{x}u_{0} + \overline{D}_{\omega}u_{1}$. The fact that $u_{1} \in L^{\Phi}\left(Q; W^{1}_{\#}L^{\Phi}(\Omega)\right)$ if $\widetilde{\Phi} \in \Delta'$  will follow from  \cite[Remark 2]{tacha2}.
		\end{enumerate}
		Let us first check (a) : Let $\mathbf{\Psi}_{\epsilon}(x,\omega) = \epsilon \varphi(x)f(T(\epsilon^{-1}x)\omega)$ for $(x,\omega) \in Q\times \Omega$ where $\varphi \in \mathcal{C}^{\infty}_{0}(Q)$ and $f \in \mathcal{C}^{\infty}(\Omega)$. Then 
		\begin{equation*}
			\iint_{Q\times \Omega} \dfrac{\partial u_{\epsilon}}{\partial x_{i}}\mathbf{\Psi}_{\epsilon} \, dxd\mu = - \iint_{Q\times \Omega} \epsilon\, u_{\epsilon} f^{\epsilon}\dfrac{\partial \varphi}{\partial x_{i}}dxd\mu - \iint_{Q\times \Omega}  u_{\epsilon}\varphi(D_{i,\omega}f^{\epsilon})dxd\mu  
		\end{equation*}
		where $f^{\epsilon}(x,\omega) = f(T(\epsilon^{-1}x)\omega)$. Letting $E' \ni \epsilon \rightarrow 0$, we get 
		\begin{equation*}
			\iint_{Q\times \Omega}  u_{0}\varphi(D_{i,\omega}f)dxd\mu = 0.
		\end{equation*} 
		Hence $\displaystyle \int_{\Omega} u_{0}(x,\cdot)D_{i,\omega}f d\mu = 0 $ for all $1\leq i \leq N$ and all $f \in \mathcal{C}^{\infty}(\Omega)$, which is equivalent to say that $u_{0}(x,\cdot) \in I_{nv}^{\Phi}(\Omega)$ for a.e. $x \in Q$. \\
		Hypothesis (i) implies that the sequence $(u_{\epsilon})_{\epsilon\in E'}$ is bounded in $W^{1}L^{\Phi}_{D_{x}}(Q ; L^{\Phi}(\Omega))$, which yields the existence of a subsequence of $E'$ not relabeled and of a function $u \in W^{1}L^{\Phi}_{D_{x}}(Q ; L^{\Phi}(\Omega))$ such that $u_{\epsilon} \rightarrow u$ in $W^{1}L^{\Phi}_{D_{x}}(Q ; L^{\Phi}(\Omega))$-weak as $E' \ni \epsilon \rightarrow 0$. In particular $\displaystyle \int_{\Omega} u_{\epsilon}(\cdot,\omega)\psi(\omega)d\mu \rightarrow \int_{\Omega} u(\cdot,\omega)\psi(\omega)d\mu$ in $L^{1}(Q)$-weak for all $\psi \in I_{nv}^{\widetilde{\Phi}}(\Omega)$. Therefore using \cite[Lemma 3.6]{bourgeat}, we get at once $u_{0} \in W^{1}L^{\Phi}_{D_{x}}(Q ; L^{\Phi}(\Omega))$, so that $u_{0} \in W^{1}L^{\Phi}_{D_{x}}(Q ; I_{nv}^{\Phi}(\Omega))$. \\
		It remains to check (b). Let $\mathbf{\Psi}_{\epsilon}(x,\omega) = \varphi(x)\mathbf{\Psi}(T(\epsilon^{-1}x)\omega)$ with $\varphi \in \mathcal{C}^{\infty}_{0}(Q)$ and $\mathbf{\Psi} = (\psi_{j})_{1\leq j\leq N} \in \mathcal{V}^{\widetilde{\Phi}}_{\textup{div}}$ (i.e. $\textup{div}_{\omega,\widetilde{\Phi}}\mathbf{\Psi} = 0$). Clearly 
		\begin{equation*}
			\sum_{j=1}^{N} \iint_{Q\times \Omega} \dfrac{\partial u_{\epsilon}}{\partial x_{j}} \varphi\, \psi_{j}^{\epsilon}\, dxd\mu = - \sum_{j=1}^{N} \iint_{Q\times \Omega} u_{\epsilon }\psi_{j}^{\epsilon}\dfrac{\partial \varphi}{\partial x_{j}}\, dxd\mu
		\end{equation*}
		where $\psi_{j}^{\epsilon}(x,\omega) = \psi_{j}(T(\epsilon^{-1}x)\omega)$. The limit passage (when $E' \ni \epsilon \rightarrow 0$) yields 
		\begin{equation*}
			\sum_{j=1}^{N} \iint_{Q\times \Omega} v_{j} \varphi\, \psi_{j}\, dxd\mu = - \sum_{j=1}^{N} \iint_{Q\times \Omega} u_{0 }\psi_{j}\dfrac{\partial \varphi}{\partial x_{j}}\, dxd\mu
		\end{equation*}
		or equivalently
		\begin{equation*}
			\iint_{Q\times \Omega} \left( \textbf{v}(x,\omega) - Du_{0}(x,\omega)  \right)\cdot \mathbf{\Psi}(\omega)\varphi(x) \, dxd\mu = 0,
		\end{equation*}
		and so, as $\varphi$ is arbitrarily fixed in $\mathcal{C}^{\infty}_{0}(Q)$,
		\begin{equation*}
			\int_{\Omega} \left( \mathbf{v}(x,\omega) - Du_{0}(x,\omega)  \right)\cdot \mathbf{\Psi}(\omega) \,d\mu = 0, \;\;\;  
		\end{equation*}
		for $\mathbf{\Psi} \in \mathcal{V}^{\widetilde{\Phi}}_{\textup{div} }$ and for a.e. $x\in Q$, where $\left( \mathbf{v} - Du_{0}  \right) \in L^{\Phi}(Q\times\Omega)^{N}$. Note that, property ix) (see Page \pageref{property9}) implies that $\left( \mathbf{v}(x,\cdot) - Du_{0}(x,\cdot)  \right) \in L^{\Phi}(\Omega)^{N}$.  Therefore, Lemma \ref{lem5} provides us with a unique $u_{1}(x, \cdot) \in  W^{1}_{\#}L^{\Phi}(\Omega)$ such that 
		\begin{equation*}
			\mathbf{v}(x,\cdot) - Du_{0}(x,\cdot) = \overline{D}_{\omega}u_{1}(x,\cdot) \;\; \textup{for} \;\, \mu\textup{-a.e.} \;\, x \in Q.
		\end{equation*} 
			Hence, $\overline{D}_{\omega}u_{1} \in L^{\Phi}(Q\times\Omega)^{N} \subset  L^{1}\left(Q; L^{\Phi}(\Omega)\right)$, (see, e.g. \cite[Page 126]{tacha3}), and so $u_{1} \in  L^{1}\left(Q; W^{1}_{\#}L^{\Phi}(\Omega)\right)$. Finally, if $\widetilde{\Phi} \in \Delta'$,  it follows that $u_{1} \in L^{\Phi}\left(Q; W^{1}_{\#}L^{\Phi}(\Omega)\right)$ (see, e.g. \cite[Remark 2]{tacha2}).  
	\end{proof}

	%%%%%%%%%%%%%%%%%%%%%%%%%%%%%%%%%%%%%%%%%%%%%%%%%%%%%%%%%%%%%%%%%%%%%%%%%%%%%%%%%%%%%
	%%%%%%%%%%%%%%%%%%%%    4eme PARTIE   %%%%%%%%%%%%%%%%%%%%%%%%%%%%%%%%%%%%%%%%%%%%%%%

	\section{Application to the homogenization of an integral functional with convex integrand and Nonstandard growth}\label{sect4}
	
	\subsection{Setting of the problem : Existence and uniqueness of minimizers}
	Let $(\Omega, \mathscr{M}, \mu)$ be a measure space with probability measure $\mu$ and $\{ T(x), \, x \in \mathbb{R}^{N} \}$ be a fixed $N$-dimensional dynamical system on $\Omega$ which is invariant for $T$. Let $\Phi \in \Delta_{2}$ be a Young function and $\widetilde{\Phi} \in \Delta_{2}$ its conjugate. \\
	Let $(\omega,\lambda) \rightarrow f(\omega,\lambda)$ be a function from $\Omega\times \mathbb{R}^{N}$ into $\mathbb{R}$ satisfying the following properties:
	\begin{itemize}
		\item[(H$_{1}$)] $f(\cdot, \lambda)$ is measurable for all $\lambda \in \mathbb{R}^{N}$ ;
		\item[(H$_{2}$)] $f(\omega,\cdot)$ is strictly convex for almost all $\omega \in \Omega$ ;
		\item[(H$_{3}$)] There are two constants $c_{1}, \, c_{2} > 0$ such that 
		\begin{equation*}
			c_{1} \Phi(|\lambda|) \leq f(\omega,\lambda) \leq c_{2}(1+ \Phi(|\lambda|))
		\end{equation*}
		for all $\lambda = (\lambda_{i})$ and for almost all $\omega \in \Omega$.
	\end{itemize}
	Let $Q$ be a bounded open set in $\mathbb{R}^{N}_{x}$ (the space $\mathbb{R}^{N}$ of variables $x = (x_{1}, \cdots , x_{N})$). \\
	For each given $\epsilon > 0$, let $F_{\epsilon}$ be the integral functional defined by 
	\begin{equation}\label{lem13}
		F_{\epsilon}(v) = \iint_{Q\times \Omega} f\left( T(\epsilon^{-1}x)\omega, \, Dv(x,\omega) \right)dxd\mu, \quad v \in W^{1}_{0}L^{\Phi}_{D_{x}}(Q ; L^{\Phi}(\Omega)),
	\end{equation} 
	where $D$ denotes the usual gradient operator in $Q$, i.e. $D = D_{x}= \left(\frac{\partial}{\partial x_{i}} \right)_{1\leq i\leq N}$ and $f$ is a random integrand satisfying above conditions $(H_{1})-(H_{3})$.  \\
	We intend to study the asymptotic behaviour (as $0 < \epsilon \rightarrow 0$) of the sequence of solutions to the problems 
	\begin{equation}\label{lem26}
		\min \left\{ F_{\epsilon}(v) : v \in W^{1}_{0}L^{\Phi}_{D_{x}}(Q ; L^{\Phi}(\Omega)) \right\}.
	\end{equation}
		Let us begin with the following lemma  whose the proof is entirely identical to the proof of \cite[Lemma 3.1]{tacha1}, and is therefore omitted here.  
	\begin{lemma}\label{lem9}
		Let $(\omega,\lambda) \rightarrow g(\omega,\lambda)$ be a function from $\Omega\times \mathbb{R}^{N}$ into $\mathbb{R}$ satisfying the following properties :
		\begin{itemize}
			\item[(i)] $g(\cdot, \lambda)$ is measurable for all $\lambda \in \mathbb{R}^{N}$ ;
			\item[(ii)] $g(\omega,\cdot)$ is strictly convex for almost all $\omega \in \Omega$ ;
			\item[(iii)] There exists  $c_{0} > 0$ such that 
			\begin{equation} \label{lem8}
				| g(\omega,\lambda)| \leq c_{0}(1+ \Phi(|\lambda|))
			\end{equation}
			for all $\lambda = (\lambda_{i})$ and for almost all $\omega \in \Omega$.
		\end{itemize}
		Then $g$ is continuous in the second argument. Precisely one has
		\begin{equation}\label{lem7}
			| g(\omega,\lambda) - g(\omega,\nu) | \leq c'_{0} \dfrac{1 + \Phi(2(1+ |\lambda| + |\nu|))}{1+ |\lambda| + |\nu|} |\lambda - \nu |,
		\end{equation}
		for all $\lambda, \nu \in \mathbb{R}^{N}$ and for almost all $\omega \in \Omega$, with $c'_{0} = 2Nc_{0}$.
	\end{lemma}
	
	Let now $\mathbf{v} = (v_{i}) \in \left[\mathcal{C}^{\infty}(\Omega)\otimes \mathcal{C}(\overline{Q}, \mathbb{R}) \right]^{N} = \mathcal{C}^{\infty}(\Omega)\otimes \mathcal{C}(\overline{Q}, \mathbb{R}) \times \cdots \times \mathcal{C}^{\infty}(\Omega)\otimes \mathcal{C}(\overline{Q}, \mathbb{R})$, $N$-times. By (\ref{lem8}), it is an easy task to check that the function $\omega \rightarrow g(\omega, \mathbf{v}(x,\omega))$ of $\Omega$ into $\mathbb{R}$ denoted by $g(\cdot, \mathbf{v}(x,\cdot))$ belongs to $L^{\infty}(\Omega)$ (we refer to \cite[Page 4]{nanguet1} for the case $\Omega \subset \mathbb{R}^{N}$). This implies that the function
	$(x, \omega) \rightarrow g(\omega, \mathbf{v}(x,\omega))$ of $Q\times\Omega$ into $\mathbb{R}$ denoted by $g(\cdot, \mathbf{v})$ belongs to $\mathcal{C}(\overline{Q}; L^{\infty}(\Omega))$. Hence for each $\epsilon > 0$, the function $(x, \omega) \rightarrow g(T(\epsilon^{-1}x)\omega, \mathbf{v}(x,\omega))$ of $Q\times\Omega$ into $\mathbb{R}$, denoted by $g^{\epsilon}(\cdot, \mathbf{v})$ is well defined as an element of $L^{\infty}(Q\times \Omega, \mathbb{R})$, and we have the following proposition and corollary.
	
	\begin{proposition}\label{lem14}
		Given $\epsilon >0$, the transformation $\mathbf{v} \rightarrow g^{\epsilon}(\cdot, \mathbf{v})$ of $\left[\mathcal{C}^{\infty}(\Omega)\otimes \mathcal{C}(\overline{Q}, \mathbb{R}) \right]^{N}$ into $L^{\infty}(Q\times \Omega, \mathbb{R})$ extends by continuity to a mapping, still denoted by $\mathbf{v} \rightarrow g^{\epsilon}(\cdot, \mathbf{v})$, of
		$\left[L^{\Phi}(Q\times \Omega, \mathbb{R}) \right]^{N}$ into $L^{1}(Q\times \Omega, \mathbb{R})$ with the property 
		\begin{equation*}\label{lem12}
			\begin{array}{rl}
				& \parallel g^{\epsilon}(\cdot, \mathbf{v}) - g^{\epsilon}(\cdot, \mathbf{w}) \parallel_{L^{1}(Q\times \Omega)} \\
				\leq & c\, \left( \|1\|_{\widetilde{\Phi},Q\times\Omega} + \parallel \phi( 1+|\mathbf{v}|+ |\mathbf{w}|)\parallel_{\widetilde{\Phi},Q\times \Omega} \right)\parallel \mathbf{v}-\mathbf{w}  \parallel_{\left[L^{\Phi}(Q\times \Omega) \right]^{N}}
			\end{array}
		\end{equation*}
		for all $\mathbf{v}, \mathbf{w} \in \left[L^{\Phi}(Q\times \Omega, \mathbb{R}) \right]^{N}$.
	\end{proposition}
	The proof of the above proposition is similar to the proof of \cite[Proposition 3.1]{tacha1}, using Lemma \ref{lem9}. 
	\begin{corollary}\label{lem22}
		Under the hypothesis $(H_{1})-(H_{3})$, given $v \in W^{1}L^{\Phi}_{D_{x}}(Q ; L^{\Phi}(\Omega))$, the function $(x,\omega) \rightarrow f(T(\epsilon^{-1}x)\omega, Dv(x,\omega))$ of  $Q\times \Omega$ into $\mathbb{R}$, denoted $f^{\epsilon}(\cdot, Dv)$ is well defined as an element of $L^{1}(Q\times \Omega, \mathbb{R})$. Moreover we have 
		\begin{equation}\label{lem15}
			c_{1} \parallel Dv \parallel_{L^{\Phi}(Q\times\Omega)^{N}} \leq \parallel f^{\epsilon}(\cdot, Dv) \parallel_{L^{1}(Q\times\Omega)} \leq c'_{2}(1 + \parallel Dv \parallel_{L^{\Phi}(Q\times\Omega)^{N}}),
		\end{equation}	
		for all $v \in W^{1}_{0}L^{\Phi}_{D_{x}}(Q ; L^{\Phi}(\Omega))$, where $c'_{2} = c_{2}\max(1, |Q|)$ with $|Q| = \int_{Q} dx$. 
	\end{corollary}
	
	We are now able to prove the existence of a minimizer on $W^{1}_{0}L^{\Phi}_{D_{x}}(Q ; L^{\Phi}(\Omega))$, for each $\epsilon > 0$ of integral functional $F_{\epsilon}$ (see (\ref{lem13})). Clearly we have the following
	\begin{theorem}
		For $\epsilon > 0$, there exists a unique $u_{\epsilon} \in W^{1}_{0}L^{\Phi}_{D_{x}}(Q ; L^{\Phi}(\Omega))$ that minimizes $	F_{\epsilon}$ on  $W^{1}_{0}L^{\Phi}_{D_{x}}(Q ; L^{\Phi}(\Omega))$, i.e.,
		\begin{equation*}
			F_{\epsilon}(u_{\epsilon}) = \min \left\{ F_{\epsilon}(v) : v \in W^{1}_{0}L^{\Phi}_{D_{x}}(Q ; L^{\Phi}(\Omega)) \right\}.
		\end{equation*}
	\end{theorem}
	\begin{proof}
		Let $\epsilon > 0$ be fixed. Thanks to Proposition \ref{lem14} (with $g=f$), there is a constant $c > 0$ such that 
		\begin{equation*}
			\begin{array}{rl}
				& \left| F_{\epsilon}(v) - F_{\epsilon}(w)  \right|\\
				\leq & c\, \left( \|1\|_{\widetilde{\Phi},Q\times\Omega} + \parallel \phi( 1+|Dv|+ |Dw|)\parallel_{\widetilde{\Phi},Q\times \Omega} \right)\parallel Dv-Dw  \parallel_{\left[L^{\Phi}(Q\times \Omega) \right]^{N}},
			\end{array}
		\end{equation*} 
		for all $v,w \in W^{1}_{0}L^{\Phi}_{D_{x}}(Q ; L^{\Phi}(\Omega))$, so that $F_{\epsilon}$ is continuous. With this in mind, since $F_{\epsilon}$ is strictly convex (see $(H_{2})$), and coercive (see the left-hand side inequality in (\ref{lem15})), there exists a unique $u_{\epsilon}$ that minimizes $F_{\epsilon}$ on $W^{1}_{0}L^{\Phi}_{D_{x}}(Q ; L^{\Phi}(\Omega))$.
	\end{proof} 
	
	\subsection{Preliminary results} 
	Let $\mathbf{v}  \in \left[\mathcal{C}^{\infty}(\Omega)\otimes \mathcal{C}(\overline{Q}, \mathbb{R}) \right]^{N}$.
	The function $\omega \rightarrow f(\omega, \mathbf{v}(x,\omega))$ of $\Omega$ into $\mathbb{R}$ denoted by $f(\cdot, \mathbf{v}(x,\cdot))$ belongs to $L^{\infty}(\Omega)$. This being, using Lemma \ref{lem9},  the function
	$(x, \omega) \rightarrow f(\omega, \mathbf{v}(x,\omega))$ of $\overline{Q}\times\Omega$ into $\mathbb{R}$ denoted by $f(\cdot, \mathbf{v})$ belongs to $\mathcal{C}(\overline{Q}; L^{\infty}(\Omega))$, with
	\begin{equation}\label{lem19}
		|f(\cdot, \mathbf{v}) - f(\cdot, \mathbf{w})| \leq c'_{2} \left(  1 + \dfrac{\Phi(2(1+ |\mathbf{v}| + |\mathbf{w}|))}{1+ |\mathbf{v}| + |\mathbf{w}|}\right) |\mathbf{v} - \mathbf{w} | \quad \textup{a.e. \, in} \, \overline{Q}\times \Omega 
	\end{equation}
	for all $\mathbf{v}, \mathbf{w} \in \left[\mathcal{C}^{\infty}(\Omega)\otimes \mathcal{C}(\overline{Q}, \mathbb{R}) \right]^{N}$. Therefore, for fixed $\epsilon >0$, one defines the function
	\begin{equation}\label{fepsilon1}
		(x, \omega) \rightarrow f(T(\epsilon^{-1}x)\omega, \mathbf{v}(x,T(\epsilon^{-1}x)\omega))\;\; \textup{of} \;\;Q\times\Omega\;\; \textup{into} \;\; \mathbb{R}\;\; \textup{denoted\; by}\; f^{\epsilon}(\cdot, \mathbf{v}^{\epsilon}),
	\end{equation}
	as an element of $L^{\infty}(Q\times \Omega, \mathbb{R})$.
	\begin{proposition}\label{lem23}
		Suppose that $(H_{1})-(H_{3})$ hold. For every $\mathbf{v} \in \left[\mathcal{C}^{\infty}(\Omega)\otimes \mathcal{C}(\overline{Q}, \mathbb{R}) \right]^{N}$ one has 
		\begin{equation}\label{lem17}
			\lim_{\epsilon \to 0} \iint_{Q\times\Omega} f(T(\epsilon^{-1}x)\omega, \mathbf{v}(x,T(\epsilon^{-1}x)\omega)) dxd\mu = \iint_{Q\times\Omega} f(\omega, \mathbf{v}(x,\omega)) dxd\mu.
		\end{equation}
		Futhermore, the mapping  $\mathbf{v} \rightarrow f(\cdot, \mathbf{v})$ of $\left[\mathcal{C}^{\infty}(\Omega)\otimes \mathcal{C}(\overline{Q}, \mathbb{R}) \right]^{N}$ into $L^{1}(Q\times \Omega, \mathbb{R})$ extends by continuity to a mapping, still denoted by $\mathbf{v} \rightarrow f(\cdot, \mathbf{v})$, of
		$\left[L^{\Phi}(Q\times \Omega, \mathbb{R}) \right]^{N}$ into $L^{1}(Q\times \Omega, \mathbb{R})$ with the property 
		\begin{equation*}
			\begin{array}{rl}
				& \parallel f(\cdot, \mathbf{v}) - f(\cdot, \mathbf{w}) \parallel_{L^{1}(Q\times \Omega)} \\
				\leq& c\, \left( \|1\|_{\widetilde{\Phi},Q\times\Omega} + \parallel \phi( 1+|\mathbf{v}|+ |\mathbf{w}|)\parallel_{\widetilde{\Phi},Q\times \Omega} \right)\parallel \mathbf{v}-\mathbf{w}  \parallel_{\left[L^{\Phi}(Q\times \Omega) \right]^{N}}
			\end{array}
		\end{equation*}
		for all $\mathbf{v}, \mathbf{w} \in \left[L^{\Phi}(Q\times \Omega, \mathbb{R}) \right]^{N}$.
	\end{proposition}
	\begin{proof}
		Let	$\mathbf{v} \in \left[\mathcal{C}^{\infty}(\Omega)\otimes \mathcal{C}(\overline{Q}, \mathbb{R}) \right]^{N}$. Since $f(\cdot, \mathbf{v})$ belongs to $\mathcal{C}(\overline{Q}; L^{\infty}(\Omega))$ and the sequence $(u_{\epsilon})$ defined by $u_{\epsilon}(x,\omega) =1$, a.e. $x\in Q$ and $\omega \in \Omega$, is weakly 2s-convergent in the mean in $L^{\Phi}(Q\times \Omega)$ to 1, the convergence result (\ref{lem17}) is a consequence of Remark \ref{lem16} ($R_{1}$). \\
		On the other hand, since $\Phi \in \Delta_{2}$, arguing as in the proof of Proposition \ref{lem14}, there is a constant $c= c(c_{2},Q,\Omega,\Phi)$ such that  
		\begin{equation*}
			\begin{array}{rl}
				&
				\parallel f(\cdot, \mathbf{v}) - f(\cdot, \mathbf{w}) \parallel_{L^{1}(Q\times \Omega)}\\
				\leq& c\, \left( \|1\|_{\widetilde{\Phi},Q\times\Omega} + \parallel \phi( 1+|\mathbf{v}|+ |\mathbf{w}|)\parallel_{\widetilde{\Phi},Q\times \Omega} \right)\parallel \mathbf{v}-\mathbf{w}  \parallel_{\left[L^{\Phi}(Q\times \Omega) \right]^{N}}
			\end{array}
		\end{equation*}
		for all $\mathbf{v}, \mathbf{w} \in \left[\mathcal{C}^{\infty}(\Omega)\otimes \mathcal{C}(\overline{Q}, \mathbb{R}) \right]^{N}$. We end the proof by routine argument of continuity and density.
	\end{proof}
	\begin{corollary}\label{lem24}
		Let 
		\begin{equation*}
			\phi_{\epsilon}(x,\omega) = \psi_{0}(x,\omega) + \epsilon \psi_{1}(x,T(\epsilon^{-1}x)\omega),
		\end{equation*}
		with $\epsilon>0$, $(x,\omega) \in Q\times \Omega$, $\psi_{0} \in \mathcal{C}^{\infty}_{0}(Q)\otimes I^{\Phi}_{nv}(\Omega)$ and $\psi_{1} \in \mathcal{C}^{\infty}_{0}(Q)\otimes \mathcal{C}^{\infty}(\Omega)$. Then 
		\begin{equation*}
			\begin{array}{rl}
				&
				\displaystyle\lim_{\epsilon \to 0} \displaystyle\iint_{Q\times\Omega} f(T(\epsilon^{-1}x)\omega, D\phi_{\epsilon}(x,\omega)) dxd\mu \\
				= & \displaystyle\iint_{Q\times\Omega} f(\omega, D_{x}\psi_{0}(x,\omega)+ D_{\omega}\psi_{1}(x,\omega)) dxd\mu.
			\end{array}
		\end{equation*}
	\end{corollary}
	\begin{proof}
		As $D\psi_{1}(x,T(\epsilon^{-1}x)\omega) = D_{x}\psi_{1}(x,T(\epsilon^{-1}x)\omega) + \frac{1}{\epsilon}D_{\omega}\psi_{1}(x,T(\epsilon^{-1}x)\omega)$ [see (\ref{lem18})], we have   $D\phi_{\epsilon}(x,\omega) = D_{x}\psi_{0}(x,\omega)+ \epsilon D_{x}\psi_{1}(x,T(\epsilon^{-1}x)\omega) + D_{\omega}\psi_{1}(x,\omega)$. Recalling that functions $D_{x}\psi_{0}, D_{x}\psi_{1}$ and $ D_{\omega}\psi_{1}$ belongs to $\left[\mathcal{C}^{\infty}(\Omega)\otimes \mathcal{C}(\overline{Q}, \mathbb{R}) \right]^{N}$, the function $f(\cdot, D_{x}\psi_{0}+ D_{\omega}\psi_{1}) $ belongs to $\mathcal{C}(\overline{Q}; L^{\infty}(\Omega))$, so that (\ref{lem17}) implies, using notation (\ref{fepsilon1}): 
		\begin{equation}\label{lem21}
			\begin{array}{rl}
				&
				\displaystyle	\lim_{\epsilon \to 0} \displaystyle \iint_{Q\times\Omega} f^{\epsilon}(\cdot, D\psi_{0}(x,\omega)+D_{\omega}\psi_{1}(x, T(\epsilon^{-1}x)\omega)) dxd\mu\\
				= & \displaystyle \iint_{Q\times\Omega} f(\omega, D_{x}\psi_{0}(x,\omega)+ D_{\omega}\psi_{1}(x,\omega)) dxd\mu.
			\end{array}
		\end{equation}
		On the other hand, since $D\phi_{\epsilon}(x,\omega) = D_{x}\psi_{0}(x,\omega)+ \epsilon D_{x}\psi_{1}(x,T(\epsilon^{-1}x)\omega) + D_{\omega}\psi_{1}(x,\omega)$, it follows from (\ref{lem19})
		\begin{equation*}
			|f^{\epsilon}(\cdot, D\phi_{\epsilon}) - f^{\epsilon}(\cdot, D_{x}\psi_{0}+D_{\omega}\psi_{1}^{\epsilon}) | \leq c\epsilon \quad \textup{in} \; Q\times\Omega, \; \epsilon >0,
		\end{equation*} 
		where $c = c\left( \|D_{x}\psi_{1}\|_{\infty}, \phi(\|D_{x}\psi_{0}\|_{\infty}), \phi(\|D_{\omega}\psi_{1}\|_{\infty})  \right) >0$. Hence 
		\begin{equation}\label{lem20}
			f^{\epsilon}(\cdot, D\phi_{\epsilon}) - f^{\epsilon}(\cdot, D_{x}\psi_{0}+D_{\omega}\psi_{1}^{\epsilon}) \rightarrow 0 \quad \textup{in} \; L^{1}(Q\times \Omega) \; \textup{as} \; \epsilon \to 0.
		\end{equation}
		The proof is completed by combining (\ref{lem21})-(\ref{lem20}) with the decomposition 
		\begin{equation*}
			\begin{array}{l}
				\displaystyle\iint_{Q\times\Omega}  f^{\epsilon}(\cdot, D\phi_{\epsilon})dxd\mu - \displaystyle\iint_{Q\times\Omega} f(\cdot, D_{x}\psi_{0}+ D_{\omega}\psi_{1})dxd\mu \\ 
				=  \displaystyle\iint_{Q\times\Omega} \left[  f^{\epsilon}(\cdot, D\phi_{\epsilon}) -  f^{\epsilon}(\cdot, D_{x}\psi_{0}+ D_{\omega}\psi_{1}^{\epsilon}) 
				\right] dxd\mu \\
				\;\;\;\; + \displaystyle \iint_{Q\times\Omega} \left[ f^{\epsilon}(\cdot, D_{x}\psi_{0}+ D_{\omega}\psi_{1}^{\epsilon}) - f(\cdot, D_{x}\psi_{0}+ D_{\omega}\psi_{1}) \right] dxd\mu. 
			\end{array}
		\end{equation*}
	\end{proof}
	
	Let  
	\begin{equation}\label{banach1}
		\mathbb{F}_{0}^{1}L^{\Phi} = W^{1}_{0}L^{\Phi}_{D_{x}}(Q ; I_{nv}^{\Phi}(\Omega))\times L^{\Phi}(Q; W^{1}_{\#}L^{\Phi}(\Omega)). 
	\end{equation}
	We equip $\mathbb{F}_{0}^{1}L^{\Phi}$ with the norm 
	\begin{equation*}
		\parallel \mathbf{u} \parallel_{\mathbb{F}_{0}^{1}L^{\Phi}} = \parallel Du_{0} \parallel_{L^{\Phi}(Q\times\Omega)^{N}} + \parallel \overline{D}_{\omega}u_{1} \parallel_{L^{\Phi}(Q\times\Omega)^{N}},\;\; \textup{with} \;\, \mathbf{u} = (u_{0}, u_{1}) \in \mathbb{F}_{0}^{1}L^{\Phi}.
	\end{equation*} 
	With this norm, $\mathbb{F}_{0}^{1}L^{\Phi}$ is a Banach space admitting,
	\begin{equation}\label{banach2}
		F^{\infty}_{0} = \left[  \mathcal{C}^{\infty}_{0}(Q)\otimes I_{nv}^{\Phi}(\Omega) \right] \times \left[ \mathcal{C}^{\infty}_{0}(Q)\otimes I_{\Phi}(\mathcal{C}^{\infty}(\Omega)) \right],	
	\end{equation}
	as a dense subspace, where $I_{\Phi}$ denotes the canonical mapping of $\mathcal{C}^{\infty}(\Omega)$ into the completion $W^{1}_{\#}L^{\Phi}(\Omega)$. \\
	Let now $\mathbf{v} = (v_{0}, v_{1}) \in \mathbb{F}_{0}^{1}L^{\Phi}$, set $\mathbb{D}\mathbf{v} = Dv_{0} + \overline{D}_{\omega}v_{1} \in L^{\Phi}(Q\times\Omega)^{N}$ and define the functional $F$ on $\mathbb{F}_{0}^{1}L^{\Phi}$ by 
	\begin{equation*}
		F(v) = \iint_{Q\times\Omega} f(\cdot, \mathbb{D}\mathbf{v})\, dxd\mu,
	\end{equation*}
	where the function $f$ here is defined as in Corollary \ref{lem22}. \\
	The hypotheses $(H_{1})-(H_{3})$ drive to the following lemma.
	\begin{lemma}
		There exists a unique $\mathbf{u} = (u_{0}, u_{1}) \in \mathbb{F}_{0}^{1}L^{\Phi}$ such that 
		\begin{equation}\label{lem28}
			F(\mathbf{u}) = \left\{  \min F(\mathbf{v})\, : \, \mathbf{v} \in \mathbb{F}_{0}^{1}L^{\Phi} \right\}.
		\end{equation}
	\end{lemma}
	
	\subsection{Regularization} 
	
	As in \cite{tacha1}, we regularize the integrand $f$ in order to get an approximating family of integrands $(f_{n})_{n\in\mathbb{N}^{\ast}}$ satisfying, in particular, properties analogous to $(H_{1})-(H_{3})$. Precisely, let $\theta_{n} \in \mathcal{C}^{\infty}_{0}(\mathbb{R}^{N})$ with $0\leq \theta_{n}$, $\textup{supp}\theta_{n} \subset \frac{1}{n}\overline{B_{N}}$ (where $\overline{B_{N}}$ denotes the closure of the open unit ball $B_{N}$ in $\mathbb{R}^{N}$) and $\int \theta_{n}(\eta)d\eta = 1$. Setting 
	\begin{equation*}
		f_{n}(\omega, \lambda) = \int \theta_{n}(\eta)f(\omega, \lambda-\eta)d\eta, \quad (\omega,\lambda) \in \Omega\times \mathbb{R}^{N}.
	\end{equation*}
	The main properties of this new integrand are the following:
	\begin{itemize}
		\item[$(H_{1})_{n}$] $f_{n}(\cdot, \lambda)$ is measurable for all $\lambda \in \mathbb{R}^{N}$ ;
		\item[$(H_{2})_{n}$] $f_{n}(\omega,\cdot)$ is strictly convex for almost all $\omega \in \Omega$ ;
		\item[$(H_{3})_{n}$] There is a constant $c_{5} > 0$ such that 
		\begin{equation*}
			f_{n}(\omega,\lambda) \leq c_{5}(1+ \Phi(|\lambda|))
		\end{equation*}
		for all $\lambda \in \mathbb{R}^{N}$ and for almost all $\omega \in \Omega$ ;
		\item[$(H_{4})_{n}$] $\dfrac{\partial f_{n}}{\partial \lambda}(\omega,\lambda)$ exists for all $\lambda \in \mathbb{R}^{N}$ and for almost all $\omega\in \Omega$, and there exists a constant $c_{6}=c_{6}(n) >0$ such that 
		\begin{equation*}
			\left| \dfrac{\partial f_{n}}{\partial \lambda}(\omega,\lambda) \right| \leq c_{6} (1+\Phi(|\lambda|)).
		\end{equation*}
	\end{itemize}
	This being so, we obtain the results in Proposition \ref{lem23} and in Corollary \ref{lem24}, where $f$ is replaced by $f_{n}$, and we have the following lemma.
	\begin{lemma}
		For every $\mathbf{v} \in \left[L^{\Phi}(Q\times\Omega,\mathbb{R})\right]^{N}$, as $n\to \infty$, one has
		\begin{equation*}
			f_{n}(\cdot,\mathbf{v}) \rightarrow f(\cdot,\mathbf{v}) \quad \textup{in} \; L^{1}(Q\times\Omega).
		\end{equation*}
	\end{lemma} 
	\begin{proof}
		Let $\mathbf{v} \in \left[L^{\Phi}(Q\times\Omega,\mathbb{R})\right]^{N}$, then 
		\begin{equation*}
			\begin{array}{rl}
				& \|f_{n}(\cdot,\mathbf{v}) - f(\cdot,\mathbf{v})  \|_{L^{1}(Q\times\Omega)}\\
				= & \displaystyle \iint_{Q\times\Omega} | f_{n}(\omega,\mathbf{v}(x,\omega)) - f(\omega,\mathbf{v}(x,\omega))
				|dxd\mu  \\
				\leq & \displaystyle \iint_{Q\times\Omega} \left[ \int \theta_{n}(\eta)| f(\omega,\mathbf{v}(x,\omega)-\eta) - f(\omega,\mathbf{v}(x,\omega))
				|d\eta  \right]dxd\mu  \\
				\leq & c'\, \displaystyle \iint_{Q\times\Omega} \left[ \displaystyle \int \theta_{n}(\eta)[ 1 + \Phi( |\mathbf{v}(x,\omega)-\eta|) + \Phi( |\mathbf{v}(x,\omega)|)]|\eta| d\eta  \right] dxd\mu \\
				\leq & \frac{c''}{n}\, \displaystyle \iint_{Q\times\Omega} \sup_{|\eta|\leq \frac{1}{n}} \left[  1 + \Phi( |\mathbf{v}(x,\omega)-\eta|) + \Phi( |\mathbf{v}(x,\omega)|)   \right] dxd\mu 
			\end{array}
		\end{equation*}
		where $c'$ and $c''$ being independent of $n$. Setting 
		\begin{equation*}
			H_{n}(x,\omega) = \dfrac{1}{n} \sup_{|\eta|\leq \frac{1}{n}} \left[  1 + \Phi( |\mathbf{v}(x,\omega)-\eta|) + \Phi( |\mathbf{v}(x,\omega)|)   \right], \;\, \textup{for} \; (x,\omega) \in Q\times\Omega,
		\end{equation*}
		with $n\geq 1$, we have $H_{n}(x,\omega) \rightarrow 0$ as $n\to \infty$. Thus the result follows by the Lebesgue dominated convergence theorem.
	\end{proof}
	
	\noindent We are now ready to prove one of the most important results of this section.
	
	\begin{proposition}
		Let $(\mathbf{v}_{\epsilon})_{\epsilon}$ be a sequence in $ \left[L^{\Phi}(Q\times\Omega,\mathbb{R})\right]^{N}$ which weakly 2s-converges (in each component) to $\mathbf{v} \in \left[L^{\Phi}(Q\times\Omega,\mathbb{R})\right]^{N}$. Then for any integer $n\geq 1$, we have that there exists a constant $C'$ such that
		\begin{equation*}
			\iint_{Q\times\Omega} f_{n}(\cdot, \mathbf{v}) dxd\mu - \dfrac{C'}{n} \leq \lim_{\epsilon \to 0} \iint_{Q\times\Omega}  f(T(\epsilon^{-1}x)\omega, \mathbf{v}_{\epsilon}(x,\omega))  dxd\mu.
		\end{equation*}
	\end{proposition}
	\begin{proof}
		Let integer $n\geq 1$ and let $(\mathbf{v}_{l})_{l}$ be a sequence in $\left[ \mathcal{C}^{\infty}_{0}(Q)\otimes\mathcal{C}^{\infty}(\Omega)\right]^{N}$ such that $\mathbf{v}_{l} \rightarrow \mathbf{v}$ in  $\left[L^{\Phi}(Q\times\Omega,\mathbb{R})\right]^{N}$ as $l\to \infty$. \\
		The convexity and the differentiability of $f_{n}(\omega,\cdot)$ imply (for any integer $l\geq 1$) 
		\begin{equation*}
			\begin{array}{rl}
				& \displaystyle \iint_{Q\times\Omega}  f_{n}(T(\epsilon^{-1}x)\omega, \mathbf{v}(x,T(\epsilon^{-1}x)\omega))  dxd\mu \\
				\geq & \displaystyle \iint_{Q\times\Omega}  f_{n}(T(\epsilon^{-1}x)\omega, \mathbf{v}_{l}(x,T(\epsilon^{-1}x)\omega))  dxd\mu  \\
				&   + \displaystyle \iint_{Q\times\Omega} \dfrac{\partial f_{n}}{\partial \lambda} (T(\epsilon^{-1}x)\omega, \mathbf{v}_{l}(x,T(\epsilon^{-1}x)\omega))\cdot \\
				& \qquad  \left( \mathbf{v}(x,T(\epsilon^{-1}x)\omega)  - \mathbf{v}_{l}(x,T(\epsilon^{-1}x)\omega) \right)  dxd\mu  
			\end{array}
		\end{equation*}
		On the other hand, taking into account $(H_{1})_{n}$, $(H_{2})_{n}$ and $(H_{4})_{n}$, the function $x\rightarrow \frac{\partial f_{n}}{\partial \lambda}(\cdot, \mathbf{v}_{l}(x, \cdot))$ of $\overline{Q}$ into $L^{\infty}(\Omega)^{N}$, denoted $\frac{\partial f_{n}}{\partial \lambda}(\cdot, \mathbf{v}_{l})$, belongs to $\mathcal{C}(\overline{Q}; L^{\infty}(\Omega)^{N})$. Hence, arguing as in the first part of the proof of Proposition \ref{lem23}, Remark \ref{lem16}$(R_{1})$ implies that 
		\begin{equation*}
			\begin{array}{rl}
				& \displaystyle\lim_{\epsilon \to 0}   \displaystyle	\iint_{Q\times\Omega} \dfrac{\partial f_{n}}{\partial \lambda} (T(\epsilon^{-1}x)\omega, \mathbf{v}_{l}(x,T(\epsilon^{-1}x)\omega))\cdot \\
				& \qquad \qquad \qquad \left( \mathbf{v}(x,T(\epsilon^{-1}x)\omega) - \mathbf{v}_{l}(x,T(\epsilon^{-1}x)\omega) \right)  dxd\mu  \\
				= & \displaystyle \iint_{Q\times\Omega} \dfrac{\partial f_{n}}{\partial \lambda} (\omega, \mathbf{v}_{l}(x,\omega))\cdot \left( \mathbf{v}(x,\omega) - \mathbf{v}_{l}(x,\omega) \right) dxd\mu. 
			\end{array} 
		\end{equation*}
		Next, we observe that for a.e. $\omega$ and every $\lambda$ and a suitable positive constant $c'_{0}$ one has 
		\begin{equation}\label{bel1}
			f_{n}(\omega,\lambda) \leq f(\omega,\lambda) + \dfrac{1}{n}c'_{0} (1 + \phi(2(1+|\lambda|))).
		\end{equation}	
		Indeed, for a.e. $\omega$, every $\lambda$, $\nu$, by (\ref{lem7}), 
		\begin{equation*}
			\begin{array}{lcr}
				f(\omega,\lambda) &  \leq &  f(\omega,\nu) +  c'_{0} \dfrac{ \Phi(2(1+ |\lambda| + |\nu|))}{1+ |\lambda| + |\nu|} |\lambda - \nu | \\
				& \leq & f(\omega,\nu) + c'_{0} (1 + \phi( 1+ |\lambda| + |\nu|))|\lambda - \nu|.
			\end{array}
		\end{equation*}	
		Replacing $\lambda$ by $\lambda-\eta$ and $\nu$ by $\lambda$ respectively, we obtain: 
		\begin{equation*}
			\begin{array}{lcr}
				f(\omega,\lambda-\eta) &  \leq &  f(\omega,\lambda) +  c'_{0} (1 + \phi( 1+ |\lambda-\eta| + |\lambda|))|\nu| \\
				& \leq & f(\omega,\lambda) + c'_{0} (1 + \phi( 1+ |\eta| + 2|\lambda|))|\nu|.
			\end{array}
		\end{equation*}
		Let $n>0$, and assume $|\eta| \leq \frac{1}{n} \leq 1$, hence, 
		\begin{equation*}
			f(\omega,\lambda-\eta)   \leq   f(\omega,\lambda) +  c'_{0} (1 + \phi(2( 1 + |\lambda|)))\frac{1}{n}. 
		\end{equation*}	
		Multiplying both side of the inequality, by $\theta_{n}$, we get: 
		\begin{equation*}
			f(\omega,\lambda-\eta)\theta_{n}(\eta)   \leq   f(\omega,\lambda)\theta_{n}(\eta) +  \frac{1}{n}c'_{0} (1 + \phi(2( 1 + |\lambda|))) \theta_{n}(\eta). 
		\end{equation*}
		Integration leads to (\ref{bel1}). Hence, given $\mathbf{v}_{\epsilon}(x,\omega)= \mathbf{v}(x, T(\epsilon^{-1}x)\omega)$, we have 	
		\begin{equation*}
			f_{n}(T(\epsilon^{-1}x)\omega,\mathbf{v}_{\epsilon}) \leq f(T(\epsilon^{-1}x)\omega,\mathbf{v}_{\epsilon}) + \dfrac{1}{n} c'_{0} (1 + \phi(2(1+|\mathbf{v}_{\epsilon}|))).
		\end{equation*}
		Thus 
		\begin{equation*}
			\begin{array}{rcl}
				\displaystyle	\iint_{Q\times\Omega} f_{n}(T(\epsilon^{-1}x)\omega,\mathbf{v}_{\epsilon}) dx d\mu & \leq & \displaystyle \iint_{Q\times\Omega} f(T(\epsilon^{-1}x)\omega,\mathbf{v}_{\epsilon}) dx d\mu  + \dfrac{1}{n}C|Q\times\Omega|  \\
				& & + \dfrac{c'_{0}}{n} \displaystyle \iint_{Q\times\Omega} \alpha \dfrac{\phi(2(1+|\mathbf{v}_{\epsilon}|))}{\alpha} dx d\mu,
			\end{array}
		\end{equation*}
		with $0 < \alpha \leq 1$.
		But 
		\begin{equation*}
			\alpha \dfrac{\phi(2(1+|\mathbf{v}_{\epsilon}|))}{\alpha} \leq \widetilde{\Phi}(\alpha\phi(2(1+|\mathbf{v}_{\epsilon}|))) + \Phi\left(\frac{1}{\alpha}\right) \leq \alpha\widetilde{\Phi}(\phi(2(1+|\mathbf{v}_{\epsilon}|))) + \Phi\left(\frac{1}{\alpha}\right).
		\end{equation*}
		Set $\Gamma_{1} = \{ (x,\omega) \in Q\times\Omega : 2(1+|\mathbf{v}_{\epsilon}|) > t_{0} \}$, $\Gamma_{2} = Q\times\Omega \backslash \Gamma_{1}$.  \\
		Hence,we get 
		\begin{equation*}
			\begin{array}{rl}
				& \displaystyle\iint_{Q\times \Omega} \alpha \dfrac{\phi(2(1+|\mathbf{v}_{\epsilon}|))}{\alpha} dx d\mu \\
				\leq & \displaystyle\iint_{Q\times \Omega} \alpha\widetilde{\Phi}(\phi(2(1+|\mathbf{v}_{\epsilon}|))) dx d\mu + \Phi\left(\frac{1}{\alpha}\right) |Q\times\Omega| \\
				\leq &  |\Gamma_{2}|\alpha \widetilde{\Phi}(\phi(t_{0})) +  \Phi\left(\frac{1}{\alpha}\right) |Q\times\Omega| + \alpha \displaystyle \int_{\Gamma_{1}} \Phi(4(1+|\mathbf{v}_{\epsilon}|)) dx d\mu.
			\end{array}
		\end{equation*}
		Let $C > 1 + \|4(1+|\mathbf{v}_{\epsilon}|)\|_{\Phi,Q\times\Omega}$. Then $\iint_{Q\times \Omega} \Phi\left(\frac{4(1+|\mathbf{v}_{\epsilon}|)}{C}\right) dx d\mu \leq 1$.  \\
		Since
		\begin{equation*}
			\begin{array}{rcl}
				\Phi(4(1+|\mathbf{v}_{\epsilon}|)) &=& \Phi\left(C\frac{4(1+|\mathbf{v}_{\epsilon}|)}{C}\right) \\
				& \leq& K(C) \Phi\left(\frac{4(1+|\mathbf{v}_{\epsilon}|)}{C}\right) \; \textup{whenever}\; \frac{4(1+|\mathbf{v}_{\epsilon}|)}{C} \geq t_{0}.
			\end{array}
		\end{equation*}
		Set $\Gamma_{3} = \left\{ (x,\omega) \in \Gamma_{1} : \frac{4(1+|\mathbf{v}_{\epsilon}|)}{C} \geq t_{0} \right\}$, $\Gamma_{4} = \Gamma_{1}\backslash\Gamma_{3}$. \\
		Hence,
		\begin{equation*}
			\begin{array}{rcl}
				\int_{\Gamma_{1}} \Phi(4(1+|\mathbf{v}_{\epsilon}|)) dx d\mu & = & \displaystyle \int_{\Gamma_{4}} \Phi(4(1+|\mathbf{v}_{\epsilon}|)) dx d\mu + \displaystyle\int_{\Gamma_{3}} \Phi(4(1+|\mathbf{v}_{\epsilon}|)) dx d\mu \\
				& \leq &  |\Gamma_{4}|\Phi(C t_{0}) + \displaystyle\int_{\Gamma_{3}} \Phi(4(1+|\mathbf{v}_{\epsilon}|)) dx d\mu  \\
				& \leq &  |\Gamma_{4}|\Phi(C t_{0}) + \alpha \displaystyle\int_{\Gamma_{3}} \Phi\left(C\dfrac{4(1+|\mathbf{v}_{\epsilon}|)}{C}\right) dx d\mu \\
				& \leq & |\Gamma_{4}|\Phi(C t_{0}) + K(C) \displaystyle\int_{\Gamma_{3}} \Phi\left(\frac{4(1+|\mathbf{v}_{\epsilon}|)}{C}\right)dx d\mu  \\
				& \leq & |\Gamma_{4}|\Phi(C t_{0}) + K(C) \displaystyle \iint_{Q\times\Omega} \Phi\left(\frac{4(1+|\mathbf{v}_{\epsilon}|)}{C}\right)dx d\mu.
			\end{array}
		\end{equation*}
		Since $\Phi \in \Delta_{2}$, and $(\mathbf{v}_{\epsilon})$ is bounded in $L^{\Phi}(Q\times\Omega)^{N}$ it results that\\ $\iint_{Q\times\Omega} \Phi\left(4(1+|\mathbf{v}_{\epsilon}|)\right)dx d\mu$ is also bounded.  
		Then we have 
		\begin{equation*}
			\begin{array}{rl}
				&\displaystyle \iint_{Q\times\Omega} f_{n}(T(\epsilon^{-1}x)\omega,\mathbf{v}_{\epsilon}) dx d\mu   \\
				\leq & \displaystyle\iint_{Q\times\Omega} f(T(\epsilon^{-1}x)\omega,\mathbf{v}_{\epsilon}) dx d\mu + \dfrac{1}{n}C|Q\times\Omega| \\
				&  + \dfrac{c'_{0}}{n} \big( \alpha |Q\times\Omega|\widetilde{\Phi}(\phi(t_{0})) + \Phi\left(\frac{1}{\alpha}\right)|Q\times\Omega|  \\
				& + \alpha\left( |\Gamma_{4}|\Phi(C t_{0}) + K(C) \right) \displaystyle \iint_{Q\times\Omega} \Phi \left(  \dfrac{4(1+|\mathbf{v}_{\epsilon}|)}{C} \right) dx d\mu    \big)  \\
				\leq & \displaystyle \iint_{Q\times\Omega} f(T(\epsilon^{-1}x)\omega,\mathbf{v}_{\epsilon}) dx d\mu + \dfrac{1}{n}C'  
			\end{array}
		\end{equation*}
		for a suitably big constant $C'$. \\
		Thus
		\begin{equation*}
			\begin{array}{l}
				\displaystyle	\liminf_{\epsilon \to 0}\displaystyle \iint_{Q\times\Omega}  f(T(\epsilon^{-1}x)\omega, \mathbf{v}(x,T(\epsilon^{-1}x)\omega))  dxd\mu  \\
				\geq  \displaystyle \iint_{Q\times\Omega}  f_{n}(\omega, \mathbf{v}_{l}(x,\omega))  dxd\mu \\
				\quad - \frac{C'}{n}  + \displaystyle \iint_{Q\times\Omega} \dfrac{\partial f_{n}}{\partial \lambda} (\omega, \mathbf{v}_{l}(x,\omega))\cdot \left( \mathbf{v}(x,\omega) - \mathbf{v}_{l}(x,\omega) \right)  dxd\mu. 
			\end{array}
		\end{equation*}	
		Using $(H_{4})_{n}$ combined with the H\"{o}lder's inequality and (\ref{lem11}) yields 
		\begin{equation*}
			\begin{array}{rl}
				& \left| \displaystyle\iint_{Q\times\Omega} \dfrac{\partial f_{n}}{\partial \lambda} (\omega, \mathbf{v}_{l}(x,\omega))\cdot \left( \mathbf{v}(x,\omega) - \mathbf{v}_{l}(x,\omega) \right)  dxd\mu \right| \\
				\leq& c'_{0} \, \|1 + \phi(|\mathbf{v}_{l}|)\|_{\widetilde{\Phi},Q\times\Omega}\cdot \|\mathbf{v}-\mathbf{v}_{l}\|_{\Phi,Q\times \Omega}.
			\end{array}
		\end{equation*}
		Since $\mathbf{v}_{l} \rightarrow \mathbf{v}$ in  $\left[L^{\Phi}(Q\times\Omega,\mathbb{R})\right]^{N}$ as $l\to \infty$, it follows that for $\delta > 0$ arbitrarily fixed, there exists $l_{0} \in \mathbb{N}$ such that 
		\begin{equation*}
			\left| \iint_{Q\times\Omega} \dfrac{\partial f_{n}}{\partial \lambda} (\omega, \mathbf{v}_{l}(x,\omega))\cdot \left( \mathbf{v}(x,\omega) - \mathbf{v}_{l}(x,\omega) \right)  dxd\mu \right| \leq \delta 
		\end{equation*}
		for all $l \geq l_{0}$.
		Hence for all $l \geq l_{0}$,
		\begin{equation*}
			\lim_{\epsilon \to 0} \iint_{Q\times\Omega}  f(T(\epsilon^{-1}x)\omega, \mathbf{v}_{\epsilon})  dxd\mu \geq \int_{Q\times\Omega} f_{n}(\omega, \mathbf{v}_{l}(x,\omega))  dxd\mu - \delta - \dfrac{C'}{n}.
		\end{equation*}
		Now sending $l \to \infty$ we have 
		\begin{equation*}
			\lim_{\epsilon \to 0} \iint_{Q\times\Omega}  f(T(\epsilon^{-1}x)\omega, \mathbf{v}_{\epsilon})  dxd\mu \geq \int_{Q\times\Omega} f_{n}(\omega, \mathbf{v}(x,\omega))  dxd\mu - \delta - \dfrac{C'}{n}.
		\end{equation*}
		Since $\delta$ is arbitrarily fixed, we are led at once to the result by letting $\delta \to 0$.
	\end{proof}
	
	Letting $n\to\infty$, and replacing $\mathbf{v}_{\epsilon}$ by $Du_{\epsilon}$, with $Du_{\epsilon}$ stochastically weakly two-scale converges in the mean (componentwise) to $\mathbb{D}\mathbf{u} = Du_{0} + \overline{D}_{\omega}u_{1}$ in $\left[ L^{\Phi}(Q\times\Omega, \mathbb{R}) \right]^{N}$, one obtains the following result : 
	\begin{corollary}\label{lem32}
		Let $(u_{\epsilon})_{\epsilon\in E}$ be a sequence in $W^{1}L^{\Phi}_{D_{x}}(Q ; L^{\Phi}(\Omega))$. Assume that $(Du_{\epsilon})_{\epsilon\in E}$ is stochastically weakly two-scale convergent in the mean componentwise to $\mathbb{D}\mathbf{u} = Du_{0} + \overline{D}_{\omega}u_{1} \in \left[ L^{\Phi}(Q\times\Omega, \mathbb{R}) \right]^{N}$, where $\mathbf{u} = (u_{0}, u_{1}) \in \mathbb{F}^{1}_{0}L^{\Phi}$ [see (\ref{banach1})]. Then 
		\begin{equation*}
			\iint_{Q\times\Omega} f(\omega, \mathbb{D}\mathbf{u}(x,\omega)) dxd\mu \leq \lim_{\epsilon \to 0} \iint_{Q\times\Omega}  f(T(\epsilon^{-1}x)\omega, Du_{\epsilon}(x,\omega))  dxd\mu.
		\end{equation*}
	\end{corollary}
	
	\subsection{Main homogenization result}
	
	Our main objective in this section is to prove the following 
	\begin{theorem}\label{lem37}
		For each $\epsilon > 0$, let $(u_{\epsilon})_{\epsilon\in E} \in W^{1}_{0}L^{\Phi}_{D_{x}}(Q ; L^{\Phi}(\Omega))$ be the unique solution of (\ref{lem26}). Then, as $\epsilon \to 0$, 
		\begin{equation}\label{lem30}
			u_{\epsilon} \rightarrow u_{0} \quad stoch.\; mean\;\textup{in} \; L^{\Phi}(Q\times\Omega)-weak \,2s  
		\end{equation}
		and 
		\begin{equation}\label{lem31}
			Du_{\epsilon} \rightarrow Du_{0} + \overline{D}_{\omega}u_{1} \quad stoch.\; mean\;\textup{in} \; L^{\Phi}(Q\times\Omega)^{N}-weak \,2s,  
		\end{equation}
		where $\mathbf{u} = (u_{0}, u_{1}) \in \mathbb{F}_{0}^{1}L^{\Phi}$ [see (\ref{banach1})], is the unique solution to the minimization problem (\ref{lem28}). 
	\end{theorem}
	\begin{proof}
		In view of the growth conditions in $(H_{3})$, the sequence $(u_{\epsilon})_{\epsilon>0}$ is bounded in    $W^{1}_{0}L^{\Phi}_{D_{x}}(Q ; L^{\Phi}(\Omega))$ and so the sequence $(f^{\epsilon}(\cdot, Du_{\epsilon}))_{\epsilon>0}$ is bounded in $L^{1}(Q\times\Omega)$. Thus, given an arbitrary fundamental sequence $E$, we get by Theorem \ref{lem29} the existence of a subsequence $E'$ from $E$ and a couple $\mathbf{u} = (u_{0}, u_{1}) \in \mathbb{F}_{0}^{1}L^{\Phi}$ such that (\ref{lem30})-(\ref{lem31}) hold when $E' \ni \epsilon \to 0$. The sequence $(F_{\epsilon}(u_{\epsilon}))_{\epsilon>0}$ consisting of real numbers being bounded, since $(u_{\epsilon})_{\epsilon>0}$ is bounded in  $W^{1}_{0}L^{\Phi}_{D_{x}}(Q ; L^{\Phi}(\Omega))$, there exists a subsequence from $E'$ not relabeled such that $\lim_{E' \ni\epsilon \to 0} F_{\epsilon}(u_{\epsilon})$ exists. \\
		It remains to verify that $\mathbf{u} = (u_{0}, u_{1})$ solves (\ref{lem28}). In fact, if $\mathbf{u}$ solves this problem, then thanks to the uniqueness of the solution of (\ref{lem28}), the whole sequence $(u_{\epsilon})_{\epsilon>0}$ will verify (\ref{lem30}) and (\ref{lem31}) when $\epsilon \to 0$. Thus our only concern here is to check that $\mathbf{u}$ solves problem (\ref{lem28}). To this end, in view of Corollary \ref{lem32}, we have 
		\begin{equation}\label{lem34}
			\iint_{Q\times\Omega} f(\omega, \mathbb{D}\mathbf{u}) dxd\mu \leq \lim_{E'\ni\epsilon \to 0} \iint_{Q\times\Omega}  f(T(\epsilon^{-1}x)\omega, Du_{\epsilon}(x,\omega))  dxd\mu.
		\end{equation}
		On the other hand, let us establish an upper bound for 
		\begin{equation*}
			\iint_{Q\times\Omega}  f(T(\epsilon^{-1}x)\omega, Du_{\epsilon}(x,\omega))  dxd\mu.
		\end{equation*}
		To do that, let $\phi = \left(\psi_{0}, I_{\Phi}(\psi_{1})\right) \in F^{\infty}_{0}$, [see (\ref{banach2})], with $\psi_{0} \in \mathcal{C}_{0}^{\infty}(Q)\otimes I_{nv}^{\Phi}(\Omega)$, $\psi_{1} \in \mathcal{C}_{0}^{\infty}(Q)\otimes \mathcal{C}^{\infty}(\Omega)$. 	Define $\phi_{\epsilon}$ as in Corollary \ref{lem24}. Since $u_{\epsilon}$ is the minimizer, one has 
		\begin{equation*}
			\iint_{Q\times\Omega}  f(T(\epsilon^{-1}x)\omega, Du_{\epsilon}(x,\omega))  dxd\mu \leq \iint_{Q\times\Omega}  f(T(\epsilon^{-1}x)\omega, \phi_{\epsilon}(x,\omega))  dxd\mu.
		\end{equation*}
		Thus, using Corollary \ref{lem24} we get 
		\begin{equation*}
			\lim_{E'\ni\epsilon \to 0} \iint_{Q\times\Omega}  f(T(\epsilon^{-1}x)\omega, Du_{\epsilon}(x,\omega))  dxd\mu \leq\iint_{Q\times\Omega}  f(\cdot, D\psi_{0} + D_{\omega}\psi_{1})  dxd\mu,
		\end{equation*}
		for any $\phi \in F^{\infty}_{0}$, and by density, for all $\phi \in \mathbb{F}_{0}^{1}L^{\Phi}$. From which we get 
		\begin{equation}\label{lem33}
			\lim_{E'\ni\epsilon \to 0} \iint_{Q\times\Omega}  f(T(\epsilon^{-1}x)\omega, Du_{\epsilon}(x,\omega))  dxd\mu \leq  \inf_{\mathbf{v} \in \mathbb{F}_{0}^{1}L^{\Phi}} \iint_{Q\times\Omega}  f(\cdot, \mathbb{D}\mathbf{v})  dxd\mu.
		\end{equation}
		Inequalities (\ref{lem34}) and (\ref{lem33}) yield 
		\begin{equation*}
			\iint_{Q\times\Omega} f(\omega, \mathbb{D}\mathbf{u}) dxd\mu = \inf_{\mathbf{v} \in \mathbb{F}_{0}^{1}L^{\Phi}} \iint_{Q\times\Omega}  f(\cdot, \mathbb{D}\mathbf{v})  dxd\mu
		\end{equation*}
		i.e. (\ref{lem28}). The proof is complete.
	\end{proof}
	\begin{remark}
		Indeed, when we consider the particular dynamical system $T(x)$ on $\Omega = \mathbb{T}^{N} \equiv \mathbb{R}^{N}/\mathbb{Z}^{N}$ (the $N$-dimensional torus) defined by $T(x)\omega = x+\omega\;\textup{mod}\;\mathbb{Z}^{N}$, then one can view $\Omega$ as the unit cube in $\mathbb{R}^{N}$ with all the pairs of antipodal faces being identified. The Lebesgue measure on $\mathbb{R}^{N}$ induces the Haar measure on $\mathbb{T}^{N}$ which is invariant with respect to the action of $T(x)$ on $\mathbb{T}^{N}$. Moreover $T(x)$ is ergodic and in this situation, any function on $\Omega$ may be regarded as a periodic function on $\mathbb{R}^{N}$ whose period in each coordinate is 1, so that in this case our integrand $f$ may be viewed as a periodic function with respect to the variable $\omega$. Therefore, stochastic problem (\ref{lem26}) is equivalent to the periodic problem (see \cite{tacha1}),
		\begin{equation*}
			\min \left\{ \int_{Q} f\left(\frac{x}{\epsilon}, Dv(x) \right)dx \; : \; v \in  W^{1}_{0}L^{\Phi}(Q) \right\}.
		\end{equation*}
	\end{remark}
	
	%\vspace{0.5cm}
	
	\bmhead{Acknowledgments} The authors would like to thank the anonymous referee for his/her pertinent remarks, comments and suggestions.
	
%	\section*{Declarations}
	
%	\begin{itemize}
%		\item Funding : No funding was received to to assist with the preparation of this manuscript.
%		\item Conflict of interest/Competing interests : We have no conflicts of interest to disclose. 
%		\item Consent to participate : All authors consented to participate in this work.
%		\item Consent for publication : We are enclosing herewith a manuscript entitled ``Stochastic two-scale convergence in the mean in Orlicz-Sobolev’s spaces and Applications to the homogenization of an integral functional" submitted to the journal ``Asymptotic Analysis" for possible evaluation. 
%		\item Ethics approval : With the submission of this manuscript we would like to undertake that the above mentioned manuscript has not been published elsewhere, accepted for publication elsewhere or under editorial review for publication elsewhere. 
%		\item Availability of data and materials : `Not applicable'
%		\item Code availability : `Not applicable'
%		\item Authors contributions : The authors contributed equally to this work. The first draft of the manuscript was written by \textsc{Tchinda Takougoum Franck} and all authors commented on previous versions of the manuscript. All authors read and approved the final manuscript.
%	\end{itemize}

	%-----------------------------------------------------Bibliographie--------------------------------------------------------------------
	\nocite{*}  %% a placer avant bibliography pour generer les references qui ne sont appelés dans le texte
	\bibliographystyle{abbrv}
	\bibliography{biblio2sm}

	% ------------------------------------------------------------------------
\end{document}